\documentclass[12pt, leqno]{amsart}
\usepackage{amssymb,amsmath}

\usepackage{float}
\usepackage{amsthm,amscd}
\usepackage{mathrsfs}
\usepackage{enumerate}
\usepackage{color}
\usepackage[all]{xy}

\newcommand{\nc}{\newcommand}
\nc{\nen}{\newenvironment}

\let\goth\mathfrak
\def\gg{\goth g}
\def\gs{\goth s}
\def\gh{\goth h}

\def\gl{\goth l}

\def\gp{\goth p}

\def\gn{\goth n}

\def\gq{\goth q}

\def\gg{\goth g}
\def\gs{\goth s}
\def\gh{\goth h}

\def\gl{\goth l}

\def\gp{\goth p}

\def\gn{\goth n}

\def\gq{\goth q}

\nc{\fg}{\mathfrak g}
\nc{\fn}{\mathfrak n}
\nc{\fm}{\mathfrak m}
\nc{\fp}{\mathfrak p}
\nc{\fh}{\mathfrak h}
\nc{\ft}{\mathfrak t}
\nc{\fk}{\mathfrak k}
\nc{\fb}{\mathfrak b}
\nc{\fs}{\mathfrak s}
\nc{\fB}{\mathfrak B}

\def\ggl{\goth{gl}}

\def\ol{\overline}

\def\Ind{\mbox{Ind}}
\def\ad{\mbox{ad}}
\def\Id{\mbox{Id}}

\def\cD{\mathcal D}

\evensidemargin =\oddsidemargin

\font\teneufm=eufm10
\font\seveneufm=eufm7
\font\fiveeufm=eufm5
\newfam\eufmfam
\textfont\eufmfam=\teneufm
\scriptfont\eufmfam=\seveneufm
\scriptscriptfont\eufmfam=\fiveeufm

\def\Z{\mathbb Z}
\def\C{\mathbb C}

\def\Z{\mathbb Z}

\nc{\vareps}{\varepsilon}

\nc{\sneg}{\operatorname{neg}}

\nen{thm}[1]{\label{#1}{\bf Theorem.\ } \em}{}
\nen{prop}[1]{\label{#1}{\bf Proposition.\ } \em}{}
\nen{lem}[1]{\label{#1}{\bf Lemma.\ } \em}{}
\nen{cor}[1]{\label{#1}{\bf Corollary.\ } \em}{}
\nen{conj}[1]{\label{#1}{\bf Conjecture.\ } \em}{}

\nen{defn}[1]{\label{#1}{\bf Definition.\ } }{}
\nen{exa}[1]{\label{#1}{\bf Example.\ } }{}
\nen{rem}[1]{\label{#1}{\em Remark.\ } }{}
\nen{exer}[1]{\label{#1}{\em Exercise.\ } }{}
\nen{sket}[1]{\label{#1}{\em Sketch of proof.\ } }{}

\nc{\Thm}[1]{Theorem~\ref{#1}}
\nc{\Prop}[1]{Proposition~\ref{#1}}
\nc{\Lem}[1]{Lemma~\ref{#1}}
\nc{\Cor}[1]{Corollary~\ref{#1}}
\nc{\Conj}[1]{Conjecture~\ref{#1}}
\nc{\Claim}[1]{Claim~\ref{#1}}
\nc{\Defn}[1]{Definition~\ref{#1}}
\nc{\Exa}[1]{Example~\ref{#1}}
\nc{\Rem}[1]{Remark~\ref{#1}}
\nc{\Note}[1]{Note~\ref{#1}}
\nc{\Quest}[1]{Question~\ref{#1}}
\nc{\Hyp}[1]{Hypoth\`ese~\ref{#1}}

\nc{\htt}{\operatorname{ht}}
\nc{\Ker}{\operatorname{Ker}}
\nc{\rker}{\operatorname{rKer}}
\nc{\im}{\operatorname{Im}}
\nc{\osp}{\mathfrak{osp}}
\nc{\sgn}{\operatorname{sgn}}
\nc{\F}{\operatorname{F}}
\nc{\Mod}{\operatorname{Mod}}
\nc{\Mat}{\operatorname{Mat}}
\nc{\Soc}{\operatorname{Soc}}
\nc{\sn}{\operatorname{sn}}

\nc{\Hom}{\operatorname{Hom}}
\nc{\End}{\operatorname{End}}
\nc{\supp}{\operatorname{supp}}
\nc{\Card}{\operatorname{Card}}
\nc{\Ann}{\operatorname{Ann}}
\nc{\Coind}{\operatorname{Coind}}
\nc{\wt}{\operatorname{wt}}
\nc{\ch}{\operatorname{ch}}
\nc{\Stab}{\operatorname{Stab}}
\nc{\Sch}{{\mathcal S}\mbox{\em ch}}
\nc{\Irr}{\operatorname{Irr}}
\nc{\Spec}{\operatorname{Spec}}
\nc{\Prim}{\operatorname{Prim}}
\nc{\Aut}{\operatorname{Aut}}
\nc{\Ext}{\operatorname{Ext}}
\nc{\Cl}{\operatorname{Cl}}

\begin{document}
\title{Bounded highest weight modules over $\gq (n)$}

\author[Maria Gorelik and Dimitar Grantcharov]{Maria Gorelik$^1$ and Dimitar Grantcharov$^2$}

\address{Department of Mathematics\\
 Weizmann Institute of Science\\
Rehovot 76100, Israel}

         \email{maria.gorelik@weizmann.ac.il}

\address{Department of Mathematics \\
         University of Texas at Arlington \\ Arlington, TX 76021, USA}

         \email{grandim@uta.edu}

\thanks{$^{1}$
Partially supported by the Minerva foundation with funding from the Federal German Ministry for Education and Research.\\
$^{2}$ This research was supported by NSA grant H98230-10-1-0207.}
\date{}

\maketitle

\begin{abstract}
A classification of the simple highest weight bounded $\gq(n)$-module is obtained. To achieve this classification we introduce a new combinatorial tool - the star action. Our result   leads in particular to a classification of all
simple weight $\gq(n)$-modules with finite dimensional weight spaces. \end{abstract}

\section*{Introduction}

It was known since the inception of the Lie superalgebras theory that some Lie superalgebra series require special consideration.  One of these series is especially interesting  due to its resemblance to the general linear Lie algebra $\gg \gl_n$ on the one hand, and because of the unique properties of its
structure and representations on the other. These are the so-called
{\it queer} (or {\it strange}) Lie superalgebras $\gq (n)$, introduced by V. Kac in~\cite{K}. The
queer nature of $\gq(n)$ is partly due to the nonabelian structure
of its Cartan subsuperalgebra $\gh$ which has  a nontrivial odd part
$\gh_{\bar{1}}$.  Because $\gh_{\bar{1}} \neq 0$, the
study of  highest weight modules of $\gq(n)$ requires  nonstandard  technique, including Clifford algebra methods. The latter is necessary due to the fact that the highest
weight space of an irreducible highest weight
$\gq(n)$-module $L(\lambda)$ has a Clifford module structure.

The representation theory of finite dimensional  $L(\lambda)$ is well developed. In \cite{Se}
A. Sergeev established several important results, including a  character formula of
$L(\lambda)$ for the so called tensor modules, i.e. submodules of
tensor powers $(\C^{n|n})^{\otimes r}$ of the natural $\gq(n)$-module $\C^{n|n}
$. The characters of all simple
finite-dimensional $\gq(n)$-modules have been found by I. Penkov and
V. Serganova in 1996 (see \cite{PS1} and \cite{PS2}) via an algorithm
using a supergeometric version of the Borel-Weil-Bott Theorem. On the other hand the character formula problem for infinite dimensional $L(\lambda)$ remains largely open.  In 2004 J. Brundan, \cite{Br},
reproved the character formula of Penkov-Serganova using a different approach and
formulated a conjecture for the characters of all $L(\lambda)$.

Important results about the simplicity of the highest weight $\gq(n)$-modules were obtained  in \cite{Go}. An equivalence of categories of strongly typical $\gq(n)$-modules and categories of $\gg \gl_n$-modules were established recently in \cite{FM}.

In this paper we  study highest weight $\gq (n)$-modules that are  {\it bounded}, i.e.  with uniformly bounded sets of weight multiplicities. The original motivation of this paper is to complete the classification of the simple weight $\gq(n)$-modules with finite weight multiplicities, i.e. those that equal the direct sum of their weight spaces, and whose weight spaces are finite dimensional. In the case of simple finite dimensional Lie algebras, by a theorem of Fernando-Futorny, \cite{Fe}, \cite{Fu}, every simple weight module is obtained by a parabolic induction from a cuspidal module, i.e. a module on which all root vectors act bijectively. The classification of simple cuspidal modules of finite dimensional simple Lie algebras was established by O. Mathieu in \cite{M}. For simple finite dimensional Lie superalgebras, a parabolic induction theorem was proved by I. Dimitrov, O. Mathieu, and I. Penkov in \cite{DMP}, where a partial classification of the simple cuspidal modules is obtained as well. Among the most interesting  cases not included in the latter classification is the case of $\gq (n)$ (or, equivalently, of the simple queer Lie superalgebra $\gp \gs \gq (n)$). Lastly, the simple cuspidal $\gq(n)$-modules are parameterized by bounded highest weight modules using localization technique, see \S \ref{sec-cusp} or \cite{Gr}.

 Another motivation to study bounded modules is that these modules come in families, such that the modules within one family can be linked one to another using a sequence of  localizations. Due to this linkage,  the modules in one family share important structural properties. In particular,  knowing the $\gg \gl_n$-decomposition of one module within a single family is sufficient to find the $\gg \gl_n$-decomposition of all remaining $\gq(n)$-modules in the family. We expect that the family consideration leading to similar $\gg \gl_n$-decompositions can be extended beyond the category of bounded modules. In addition to their localization description, the families have nice geometric realizations in terms of $\mathcal D$-modules. This was noticed in the case of $\gg \gl_n$ in \cite{M} and is partly explored  in the case of $\gq(n)$ in this paper.

 The main tool we use for the classification of the highest weight  bounded $\gq (n)$-modules is an analog of the dot action of the Weyl group, called in the paper the {\it star action}. The star action
 is a mixture of the regular action and the dot action depending on the atypicality of the weight. More precisely, for a simple root $\alpha$ we set
$$s_{\alpha}*\lambda=\left\{
\begin{array}{ll}
s_{\alpha}\lambda &\text{ if } \lambda (h_{\ol{\alpha}}) \not=0,\\
s_{\alpha}\cdot\lambda  &\text{ if } \lambda (h_{\ol{\alpha}}) =0,
\end{array}
\right.$$
where $s_{\alpha}\cdot\lambda:=s_{\alpha}(\lambda+\rho_0)-\rho_0$ is the standard twisted action.
This new action involves a group $\widetilde{W}$ of Coxeter type.

Recall that a $\ggl_n$-module
$\dot{L}(\lambda)$ is finite-dimensional if and only if for each simple root $\alpha$ one has $s_{\alpha}\cdot\lambda<\lambda$
 (with respect to the standard partial, see~\S~\ref{gl-roots}); by contrast,
the $\gq(n)$-module $L(\lambda)$ is finite-dimensional if and only if for each simple root $\alpha$ one has $s_{\alpha}*\lambda<\lambda$.
For an integral weight $\mu$ the simple $\ggl_n$-module $\dot{L}(\mu)$ is bounded  if and only if
the following conditions hold

(i) there exists a unique increasing ``$W$-string'' (see~\S\ref{dotiii}) $\mu=\mu_0<\mu_1<\mu_2<\ldots<\mu_s$;

(ii) the set $\{i: s_i\cdot\mu_j=\mu_j\}$ is empty for  $j<s$ and has cardinality at most one for $j=s$.

We  show that the same description for bounded weights is valid for $\gq(n)$ if we change the dot action by the $*$-action.

Our main result in terms of the star action states, roughly, that every $\gq(n)$-bounded weight $\lambda$ (i.e., for which $L(\lambda)$ is bounded) can be obtained by applying the star action of  the product $s_i...s_j$ on a ``maximal'' weight, where $s_i$ is the Weyl reflection corresponding to the $i$th simple root, and the product equals $s_i s_{i+1}...s_j$ for $i\leq j$ and  $s_i s_{i-1}...s_j$, otherwise.  The choice of a maximal weight is similar to the one in the case of $\gg \gl_n$ and depends on the type of the family (regular integral, singular, or nonintegral). In view of this, our result can be considered as the queer analog of the description of the $\gg \gl_n$-bounded weights. Combining the star action and the dot action, and using localizaton technique, one can obtain $\gg \gl_n$-decomposition factors of a bounded module $L(\lambda)$. For some special $\lambda$, in \S  \ref{sec_gl_n}, we obtain all  $\gg \gl_n$-decomposition factors.

It is worth noting that the paper, with the exception of \S \ref{sec_gl_n}, is self contained, with most of the results proved  or partially proved  both for  $\gg \gl_n$ and $\gq(n)$. The description of the bounded weights of $\gq (n)$ requires a careful analysis of the orbits of the action of the group $\widetilde{W}$ - something which is less challenging in the case of $\gg \gl_n.$

The organization of the paper is as follows. In Section 2  we include and prove important results for the classification of the $\gg \gl_n$-bounded weights in terms of the dot action of the Weyl group. Section 3 is devoted to preliminaries on the localization functor and some important $\gq (3)$-considerations. In Sections 4 and 5 we introduce the star action of the group $\widetilde{W}$ and  prove our main classification result. In the next section we study the $\gg \gl_n$-structure of bounded $\gq(n)$-modules. The $\gg \gl_n$-decomposition factors of the bounded $\gq(n)$-modules in some particular families are found in Section 7. In Section 8 we use our main result to complete the classification of all simple cuspidal (and hence of all simple weight) $\gq(n)$-modules.

\medskip
Both authors would like to express their gratitude for the excellent working conditions provided by Max Planck Institute for Mathematics, Bonn, and by Mathematisches Forschungsinstitut Oberwolfach where most of this research was performed. We would like to thank V. Mazorchuk and V. Serganova for the fruitful discussions. We are also grateful to the referee for the numerous helpful suggestions.

\section{Preliminaries}
Our ground field is $\mathbb C$.
\subsection{Notation for $\gg \gl_n$}\label{gl-roots}
We choose the natural triangular decomposition of $\ggl_n$ (the Cartan subalgebra $\gh_{\bar{0}}$ which consists of the diagonal matrices; 
and nilpotent subalgebras $\gn^{\pm}_{\bar{0}}$ which consists of strictly upper (resp., lower) triangular matrices).
We denote by $\Delta:= \Delta (\ggl_n, \gh_{\bar{0}})$ the corresponding root system. We fix a basis $\{ \varepsilon_1,...,\varepsilon_n\}$ of $\gh_{\bar{0}}$ such that
$\Delta = \{ \varepsilon_i -  \varepsilon_j \; | \; i \neq j \}$; then $\Pi=\{ \varepsilon_i -  \varepsilon_{i+1} \; | \; i =1,...,n-1 \}$ is the set of simple roots.
 For every root $\alpha \in \Delta$ we fix a standard $\gs \gl_2$-triple $(e_{\alpha}, f_{\alpha}, h_{\alpha})$ such that $e_{\alpha}$ is in the $\alpha$-root space  of $\ggl_n$, and $f_{\alpha} := e_{-\alpha}$.  We denote $\varepsilon_i -  \varepsilon_{i+1} $ by $\alpha_{i}$. Set $e_i := e_{\alpha_i}$ and $f_i := f_{\alpha_i}$.

We set $Q^+:=\sum_{\alpha\in\Pi}\mathbb{Z}_{\geq 0}\alpha$
 and introduce the standard partial order on $\gh_{\bar{0}}^*$: $\mu\leq \nu$ if $\nu-\mu\in Q^+$.

We denote the standard bilinear form on $\gh^*_{\bar{0}}$ by $(\cdot, \cdot)$. A weight
$\lambda_1 \vareps_1 + ...+ \lambda_n \vareps_n\in\gh^*_{\bar{0}}$ will be often denoted by $(\lambda_1,...,\lambda_n)$ for convenience.

We denote the Weyl group of $\ggl_n$ by $W$. A reflection in $W$ corresponding to a root $\alpha$ will be denoted by $s_{\alpha}$.  Set $s_i:=s_{\alpha_i}$. For $1 \leq i,k\leq n-1$ we will use the following convention.

$$\prod_{j=i}^k s_j:=\left\{\begin{array}{ll}
s_is_{i+1}\ldots s_k &\ \text{ if } k\geq i,\\
s_is_{i-1}\ldots s_k & \ \text{ if } k<i.
\end{array}
\right.$$

The length of an element $w$ of $W$, i.e. the number of simple reflections in a reduced expression of $w$, will be denoted by $l (w)$. By $\rho$ we denote the half sum of the positive roots. For $w \in W$ and $\lambda \in \gh_{\bar{0}}^*$, we set $w \cdot \lambda = w(\lambda + \rho) - \rho$.

 For a $\ggl_n$--module $M$ we say that a root element $e_{\alpha}$ acts {\it injectively} on $M$
if  $e_{\alpha} m \neq 0$ for every non-zero $m \in M$;
we say that $e_{\alpha}$ acts {\it locally finitely} on $M$ if for every $m \in M$
there is a positive integer $N$ such that $e_{\alpha}^N m = 0$.

\subsubsection{}\begin{defn}{} \label{def-weight} Let $M$ be a  $\ggl_n$-module.

(i) For $\lambda \in \gh_{\bar{0}}^*$, the space $M^{\lambda} = \{ m \in M \; | \; hm = \lambda(h)m, \mbox{ for every } h \in \gh_{\bar{0}}^*\}$ is called {\em the $\lambda$-weight space of $M$}.

(ii) We call $M$ a {\em weight module} if $M = \bigoplus_{\lambda \in \gh_{\bar{0}}^*} M^{\lambda}$ and $\dim M^{\lambda} < \infty$ for every $\lambda \in  \gh_{\bar{0}}^*$.

(iii) We call $M$ a {\em bounded module} if $M$ is a weight module and there is a constant $C$ such that $\dim M^{\lambda} < C$ for every $\lambda \in  \gh_{\bar{0}}^*$. We call $\max_{\lambda \in \gh_{\bar{0}}^*} \dim M^{\lambda}$  {\em the degree of $M$}.

(iv)  We call $M$ a {\em cuspidal module} (or, {\em torsion free}) if  $M$ is a weight module and  $e_{\alpha}$ acts injectively on $M$ for every $\alpha \in \Delta$.
\end{defn}

\begin{rem}{}
(i) The condition ``$\dim M^{\lambda} < \infty$'' in (ii) is added for convenience. We note that many authors consider weight modules with possibly infinite weight multiplicities.

(ii) By  Definition \ref{def-weight} (iv), it is clear that every finitely generated  cuspidal module is bounded.

(iii) Every finitely generated bounded module has finite length by Lemma 3.3 in  \cite{M}.
\end{rem}

\subsubsection{} \label{gl-n-not}
 Denote by  $\dot{M}(\lambda)$ and $\dot{L}(\lambda)$  the  Verma $\ggl_n$--module of weight $\lambda$
and its simple quotient, respectively.  We call $\lambda\in\gh^*_{\bar{0}}$ a {\it $\ggl_n$--bounded weight},
if $\dot{L}(\lambda)$ is bounded.

A weight $\ggl_n$--module $M$ is said to {\it have a shadow} if for every root $\alpha$,
$e_{\alpha}$ acts either injectively or locally finitely on $M$.
For every module $M$ and a root $\alpha$, the elements $m$
on which $e_{\alpha}$ acts locally finitely form a submodule of $M$.
In particular, every simple $\ggl (n)$--module has a shadow. If $M$ has a shadow
denote by $\Delta^{\rm{inj}}M$ (respectively, $\Delta^{\rm{fin}}M$) all roots $\alpha$
such that $e_{\alpha}$ acts injectively (respectively, locally finitely) on $M$.
The pair $( \Delta^{\rm{inj}}M, \Delta^{\rm{fin}}M)$ will be called the {\it shadow} of $M$.
The shadow of simple bounded highest weight modules can be easily described using sets
of roots of cominiscule parabolic subalgebras of $\ggl_n$. For every $k$, $1 \leq k \leq n-1$,
let $\gp_k = \gs_k \oplus \gn_k$ be the parabolic subalgebra of $\ggl_n$ with Levi part
$\gs_k$ and nilpotent radical $\gn_k$ whose root systems are given by
\begin{eqnarray*}
\Delta_{\gn_k}&=& \{ \varepsilon_i - \varepsilon_j \; | \; i \leq k <j \} \\
\Delta_{\gs_k}&=& \{ \varepsilon_i - \varepsilon_j \; | \; i, j \leq k \mbox{ or } i,j>k\} \setminus \{ 0\}.
\end{eqnarray*}
Then the  radical $\gn_k^{-}$ of the opposite parabolic subalgebra will have root system
$\Delta_{\gn^{-}_k} = - \Delta_{\gn_k}$.  Let also $\gp_1'=\gp_1$, $\gp_k' = \gp_k \cap \gp_{k-1}$,
$k=2,...,n-2$, $\gp_n'=\gp_{n-1}$, and denote by $\gs_i'$ and $\gn_i'$
the corresponding Levi parts and nilpotent radicals of the parabolic subalgebras $\gp_i'$.
We have the following proposition proved in \cite{M}.

\subsubsection{}\begin{prop}{prop1} \label{gl-bounded} Let $\dot{L}(\lambda)$ be an infinite dimensional bounded module.

(i) If $\lambda$ is an integral weight, then there is a unique $k$, $1 \leq k \leq n-1$,
such that $\Delta^{\rm{inj}} \dot{L}(\lambda) = - \Delta_{\gn_k}$ and
$\Delta^{\rm{fin}} \dot{L}(\lambda) =  \Delta_{\gs_k} \sqcup \Delta_{\gn_k} $.

(ii) If $\lambda$ is nonintegral, then there is a unique $k$, $1 \leq k \leq n$, such that
 $\Delta^{\rm{inj}} \dot{L}(\lambda) = - \Delta_{\gn'_k}$ and
$\Delta^{\rm{fin}} \dot{L}(\lambda) =  \Delta_{\gs_k'} \sqcup \Delta_{\gn_k'} $

\end{prop}
\subsubsection{}\begin{defn}{def-gl-type}
We will call a $\gg \gl_n$-bounded integral weight $\lambda$ to be {\it of type $k$} if it satisfies the conditions of (i) in~\Prop{prop1}. A $\gg \gl_n$-bounded nonintegral weight $\lambda$ which satisfies  the conditions of (ii) in~\Prop{prop1} will be called {\it  of type $1$} if $k=1$, {\it of type $(k,k-1)$} if $1<k<n$, and {\it of type $n-1$} if $k=n$. The type $1$ and type $n-1$  $\gg \gl_n$-bounded weights of  are of special interest. In these cases cases, $\dot{L}(\lambda)$ is  a quotient of parabolically induced modules for $\gp_1$ and $\gp_{n-1}$, respectively.\end{defn}

\subsection{Notation for $\gq (n)$}\label{notation-qn}
 Recall that $\gg := \gq(n)$ is the Lie subalgebra of $\gg \gl (n|n)$ consisting of all matrices of the form $X_{A,B}:=\left( \begin{matrix}  A & B \\ B & A  \end{matrix}  \right)$. The even part of $\gq(n)$ is naturally isomorphic to $\gg \gl_n$. We choose the natural triangular decomposition:
$\gg=\gn^-\oplus\gh\oplus\gn^+$ where $\gh_{\ol{0}}$ consists of the elements
$X_{A,0}$ where $A$ is diagonal, $\gh_{\ol{1}}$ consists of the elements
$X_{0,B}$ where $B$ is diagonal,  and $\gn^+$ (resp., $\gn^-$)
consists of the elements $X_{A,B}$ where $A,B$ are
strictly upper-triangular (resp., lower-triangular).
 
The root system of $\gq (n)$ is $\Delta$, i.e. coincides with the one of $\ggl_n$, however each root space $\gq(n)^{\alpha}$ has both even and odd dimension $1$.

For every $\alpha \in \Delta$ we fix odd generators $E_{\alpha}$ of $\gq(n)^{\alpha}$ and $F_{\alpha}$ of $\gq(n)^{-\alpha}$ with $F_{\alpha} = E_{-\alpha}$.

We denote by $M(\lambda)$, $N(\lambda)$ and $L(\lambda)$ the Verma $\gq(n)$--module with highest
weight $\lambda$, the Weyl module of highest weight $\lambda$, and the unique simple quotient of $N(\lambda)$,
respectively. Unless otherwise stated, for a highest weight module $N$, by  $\ol{N}$
we will denote the maximal submodule of $N$ that intersects trivially the highest weight space of $N$.

\subsubsection{}
A natural question is whether for $\lambda\not=0$ one has
$L(\lambda)_{\ol{0}}\cong L(\lambda)_{\ol{1}}$ as $\mathfrak{gl}_n$-modules
or, at least, whether the following formula holds:
\begin{equation}\label{L01}
 \ch L(\lambda)_{\ol{0}}=\ch L(\lambda)_{\ol{1}}
\end{equation}

Up to isomorphism, there exists at most two simple modules of highest weight $\lambda$:
$L(\lambda)$ and $\Pi(L(\lambda))$, where $\Pi$ is the parity change functor.
Let $r$ be the number of non-zero coordinates of $\lambda$:
 $\lambda=(a_1,\ldots,a_n)$, $r:=\#\{i: a_i\not=0\}$.
One  has $L(\lambda)\cong \Pi(L(\lambda))$ if and only if
$r$ is odd, see~\cite{P}.
If $r$ is even, one has $L(\lambda)^*\cong L(\lambda)$ if $r$ is divisible by $4$ and
$L(\lambda)^*\cong \Pi(L(\lambda))$ otherwise.
In particular, the answer on the first question is positive if $r$ is odd and the formula~(\ref{L01})
holds if $r$ is not divisible by $4$.

It is not hard to deduce from~\Lem{lemcliff} below that
the formula~(\ref{L01}) also holds for all $\lambda$ apart for the cases when
there exists a sequence $\lambda=\lambda^0,\lambda^1,\ldots,\lambda^k=0$, where
$L(\lambda^{i+1})$ is a subquotient of $M(\lambda^i)$; note that if such a sequence
exists, then $\lambda\in Q^+$  and for each $k$
one has $\#\{i: a_i=k\}=\#\{i: a_i=-k\}$.

\subsubsection{}
In~\Lem{deg-pos} we will use the following fact.

\begin{lem}{lemcliff}
For any $\lambda\not\in Q^+$ one has
$\ch L(\lambda)_{\ol{0}}=\ch L(\lambda)_{\ol{1}}$.
\end{lem}
\begin{proof}
Take any weight module $N$ and consider
its weight space  $N^{\nu}$. Clearly, $N^{\nu}$ is a $U(\fh)$ module
and its annihilator contains the ideal $J_{\nu}$ generated by
the elements  $h-\langle\nu,h\rangle, h\in\fh_{\ol{0}}$.
One readily sees that as an associative superalgebra
$U(\fh)/J_{\nu}$ is isomorphic to a product
of a Clifford superalgebra $\Cl(r)$ and the external algebra $\Lambda(n-r)$,
where $r$ is the number of non-zero entries in $\nu=(a_1,\ldots,a_n)$
and the $\mathbb{Z}_2$-grading  on the external algebra is induced
by its $\mathbb{Z}$-grading. If $r\not=0$, then for  any simple $\Cl(r)$-module $E$
one has $\dim E_{\ol{0}}=\dim E_{\ol{1}}$. Therefore for $\nu\not=0$
any simple $U(\fh)/J_{\nu}$-module has the same property
and thus any  $U(\fh)/J_{\nu}$-module has the same property.
Therefore for any weight module $N$ one has
$$\forall\nu\not=0\ \ \dim (N^{\nu}\cap N_{\ol{0}})=\dim (N^{\nu}\cap N_{\ol{1}}).$$
If $\lambda\not\in Q^+$, then $L(\lambda)_0=0$
and the assertion follows.
\end{proof}

\subsubsection{Definitions} \label{def-qn}
We call a $\gq(n)$--module {\it weight}, {\it bounded}, or {\it cuspidal}  if, viewed as  $\ggl (n)$--module,
 it is a weight, bounded, or cuspidal, respectively.

\begin{rem}{}
The definitions of ``cuspidal''  and ``torsion free'' modules of a classical Lie superalgebra $\gg$ are
different from the one used here. In general, a weight module $M$ is called cuspidal if it is not parabolically induced. In the particular case of
$\gg = \gq (n)$ and a simple module $M$,  the two notions coincide (see \cite{DMP} for details).
\end{rem}

We call $\lambda\in\gh^*_{\bar{0}}$ a {\it $\gq(n)$--bounded ({\rm or just} bounded) weight}, if $L(\lambda)$ is bounded.

Introduce a partial order on $\C$ as follows:
$$
z \succ w \mbox{ if } z-w\in\mathbb{Z}_{>0} \mbox{ or } z=w=0.
$$
By~\cite{P},
$L(\lambda)$ is finite dimensional if and only if
$(\lambda,\vareps_1) \succ (\lambda,\vareps_2)\succ...\succ (\lambda,\vareps_n)$.

For a root $\alpha = \vareps_i - \vareps_j$ we set $\ol{\alpha} = \vareps_i + \vareps_j$.
We say that a weight $\lambda$ is {\it $\alpha$-atypical} if $(\lambda, \ol{\alpha}) = 0$.
If $(\lambda, \ol{\alpha}) \neq 0$ for all $\alpha \in \Delta$ we call $\lambda$
{\it typical}. If $(\lambda, \alpha) \in \Z$ we say that $\lambda$ is
{\it $\alpha$--integral}.

A $\gq(n)$--module is said to {\it have a shadow} if, viewed as a $\ggl_n$--module, it has a shadow.

As it will be seen in Section \ref{sing_arbitrary}, the notion of singularity for arbitrary weights
of $\gq (n)$ is ambiguous. However, one has a well-defined singularity notion for bounded weights.

\subsubsection{}\label{fam_bounded}
We finish this section with an important example of a family of bounded $\gq (n)$-modules.

Let $\mathcal W ({\bf x}, {\boldsymbol{\xi}})$ denote the superalgebra of differential operators of the polynomial superalgebra $\C [x_1,...,x_n; \xi_1,...,\xi_n]$, where the
$x_i$'s are even,  the $\xi_j$'s are odd (and $\xi_j^2 = 0$).
We view $\C [x_1,...,x_n; \xi_1,...,\xi_n]$ as a $\Z$-graded ring with $\deg x_i = \deg \xi_j = 1$.

The correspondence
$$
 \left( \begin{matrix}  A & B \\ B & A  \end{matrix}  \right) \mapsto \sum_{i,j=1}^n \left( a_{ij}x_{i} \frac{\partial}{ \partial x_j} +  a_{ij}\xi_{i} \frac{\partial}{ \partial \xi_j}  + b_{ij} x_{i} \frac{\partial}{ \partial \xi_j} + b_{ij} \xi_{i} \frac{\partial}{ \partial x_j} \right)
 $$
 is a homomorphism of Lie superalgebras. For every ${\boldsymbol{\mu}} = (\mu_1,...,\mu_n) \in \C^n$ consider the space
$$
{\mathcal F}_{\boldsymbol{\mu}} = \{ f \in x_1^{\mu_1}...x_n^{\mu_n} \C [x_1^{\pm 1},...,x_n^{\pm 1}; \xi_1,...,\xi_n] \; | \; \deg f = \mu_1 +...+\mu_n\},
$$
where $\deg f$ is determined by the above convention $\deg x_i=\deg \xi_i=1$.
The above correspondence endows ${\mathcal F}_{\boldsymbol{\mu}}$ by a structure of $\gq(n)$-module.
One readily sees that ${\mathcal F}_{\boldsymbol{\mu}}$ is a bounded $\gq(n)$-module with $\deg {\mathcal F}_{\boldsymbol{\mu}} = 2^n$ (see Definition \ref{def-weight} (iii)). One has ${\mathcal F}_{\boldsymbol{\mu}} = {\mathcal F}_{\boldsymbol{\eta}}$ if and only if $\mu_i - \eta_i \in \Z$ and $\mu_1+...+\mu_n = \eta_1+...+\eta_n$. For $c \in \C$ we set for convenience ${\mathcal F}_c :={\mathcal F}_{(c,0,...,0)}$.

\section{Localization of weight $\gq(n)$--modules}

\subsection{The localization functor} \label{loc-functor} In this subsection we recall the definition of
the localization functor of weight modules.
For details we refer the reader to \cite{De} and \cite{M}.

Denote by $U$ the universal enveloping algebra $U(\gq (n))$ of $\gq(n)$.
For every $\alpha \in \Delta$ the multiplicative set
${\bf F}_{\alpha}:=\{ f_\alpha^n \; | \; n \in \Z_{\geq 0} \} \subset U$
satisfies Ore's localization conditions
because $\ad f_\alpha$ acts locally finitely on $U$. Let
$\cD_\alpha U$ be the localization of $U$ relative to ${\bf F}_{\alpha}$.
For every weight module $M$ we denote by $\cD_\alpha M$ the {\it
$\alpha$--localization of $M$}, defined as $\cD_\alpha M =
\cD_\alpha U \otimes_U M$. If $f_\alpha$ is injective on $M$, then $M$ can be naturally viewed as
a submodule of $\cD_\alpha M$, $f_\alpha$ is injective on
$\cD_\alpha M$, and $\cD_\alpha^2 M = \cD_\alpha M$. Furthermore, if
$f_\alpha$ is injective on $M$, then it is bijective on $M$ if and
only if $\cD_\alpha M = M$. Finally,
if $[f_\alpha, f_\beta] = 0$ and both $f_\alpha$ and $f_\beta$ are
injective on  $M$, then $\cD_\alpha \cD_\beta M = \cD_\beta
\cD_\alpha M$.

\subsection{Generalized conjugations.} \label{gencon}
 For $x \in \C$ and $u \in \cD_\alpha U$  we set
\begin{equation} \label{theta}
\Theta_x(u):= \sum_{i \geq 0} \binom{x}{i}\,
( \ad f_\alpha)^i (u) \, f_\alpha^{-i},
\end{equation}
where $\binom{x}{i}= \frac{x(x-1)...(x-i+1)}{i!}$. Since $\ad f_\alpha$ is locally nilpotent on
$U_\alpha$, the sum above is
actually finite. Note that for $x \in \Z$ we have $\Theta_x(u) =
f_\alpha^x u f_\alpha^{-x}$.  For a $\cD_\alpha U$-module $M$ by
$\Phi_\alpha^x M$ we denote the $\cD_\alpha U$-module $M$ twisted by
the action
$$
u \cdot v^x := ( \Theta_x (u)\cdot v)^x,
$$
where $u \in \cD_\alpha U$, $v \in M$, and $v^x$ stands for the
element $v$ considered as an element of $\Phi_\alpha^x M$. In
particular, $v^x \in M^{\lambda + x \alpha}$ whenever $v \in
M^\lambda$. Since $v^n = f_{\alpha}^{-n} \cdot v$ whenever $n \in
\Z$ it is convenient to set $f_{\alpha}^x \cdot v :=v^{-x}$ in $\Phi_\alpha^{-x} M$ for $x
\in \C$.

\subsection{} \label{lmnew} The following lemma is straightforward.

\begin{lem}{}
Let $M$ be a $\cD_{\alpha} U$-module, $v \in M$, $u \in \cD_\alpha U$ and $x,y \in \C$. Then

\noindent {\rm (i)} $\Phi_\alpha^x  M \simeq M$ whenever $x \in \Z$.

\noindent {\rm(ii)} $\Phi_\alpha^x (
\Phi_\alpha^y M) \simeq \Phi_\alpha^{x+y}M $ and, consequently, $\Phi_\alpha^x
\circ\Phi_\alpha^{-x} M \simeq \Phi_\alpha^{-x} \circ\Phi_\alpha^{x} M \simeq M$.

\noindent {\rm(iii)} $f_{\alpha}^x \cdot (f_{\alpha}^y \cdot v) =
f_{\alpha}^{x+y} \cdot v$.

\noindent {\rm(iv)} $f_{\alpha}^x \cdot (u \cdot (f_{\alpha}^{-x}
\cdot v)) = \Theta_x(u) \cdot v$.

\end{lem}

\medskip

In what follows we set $\cD_\alpha^x M:=\Phi_\alpha^x (\cD_\alpha M)$ and refer to it as a
{\it twisted localization of $M$}. Note that the localization and the twisted localization functors are exact.

\subsection{Example: the case $\gq (2)$}\label{q2} In this case all weights are bounded.
We have one simple root $\alpha$, even root vectors $e_{\alpha}, f_{\alpha}$,
and odd ones $E_{\alpha}, F_{\alpha}$.
The $\ggl_2$-decomposition of the simple highest weight $\gq(2)$--modules is the following:
$$L(\lambda)=\left\{
\begin{array}{ll}
\dot{L}(\lambda)^{\oplus r}& \text{ for } (\lambda,\alpha)=1 \mbox{ or } (\lambda,\ol{\alpha})=0;\\
 \dot{L}(\lambda)^{\oplus 2 } \oplus \dot{L}(\lambda-\alpha)^{\oplus 2} & \text{ for } (\lambda,\alpha)\in\mathbb{Z}_{>1}\ \mbox{ and }  (\lambda,\ol{\alpha})\not=0;\\
N(\lambda) &\text{ otherwise}.
\end{array}
\right.$$
where $r \in \{ 1,2 \}$ (see, for example,~\cite{Maz}).

In particular, $L(\lambda)$ is finite-dimensional if and only if $(\lambda,\alpha)\in\mathbb{Z}_{>0}$
or $\lambda=0$ that is $(\lambda,\alpha)=(\lambda,\ol{\alpha})=0$.

Below we describe some relations between certain weights in terms of twisted localization.
\subsubsection{}\label{gl2-q2-int} Assume that $\lambda$ is integral with $(\lambda,\ol{\alpha})\not=0$ (integral typical case).
Then the module $N(\lambda)$
is either simple (for $(\lambda,\alpha)\not\in\mathbb{Z}_{>0}$) or has length two with the submodule
$N(s_{\alpha}\lambda)=L(s_{\alpha}\lambda)$ and the finite-dimensional quotient $L(\lambda)$.
In the latter case, if $v$ is a highest weight vector in $N(\lambda)$, then
$f_{\alpha}^{(\lambda, \alpha)+ 1}v\in L(s_{\alpha}\lambda)$ and $f_{\alpha}^{(\lambda, \alpha)+ 1}v$ is a $\gg \gl_2$-highest weight vector.  Conversely, if $u$ is a $\gg \gl_2$-highest weight vector in $L(s_{\alpha}\lambda)$ of weight $s_{\alpha}\lambda - \alpha$, then $f_{\alpha}^{-(\lambda, \alpha)- 1}u$ is a $\gq(2)$-highest weight vector in $\cD_{\alpha}L(s_{\alpha}\lambda)$. In particular,
$L(\lambda)$ is a subquotient of the localization
$\cD_{\alpha}L(s_{\alpha}\lambda)$ if $ \lambda \geq s_{\alpha}\lambda$.

\subsubsection{} Consider now the case of integral $\lambda$ and $(\lambda,\ol{\alpha})=0$
(integral atypical case). The module $N(\lambda)$  as a $\ggl_2$-module
is the direct sum $(\dot{M}(\lambda)\oplus\dot M(\lambda-\alpha))^{\oplus r}$ and one readily sees that
$\dot{M}(\lambda)^{\oplus r}$ is, in fact, a $\gq (2)$-submodule of $N(\lambda)$.
If $(\lambda,\alpha)\in\mathbb{Z}_{\geq 0}$,
this submodule has length two: it has a submodule isomorphic to
$L(s_{\alpha} \cdot \lambda)=\dot{L}(s_{\alpha} \cdot \lambda)^{\oplus r}$
and the quotient is $L(\lambda)=\dot{L}(\lambda)^{\oplus r}$. Therefore
$L(\lambda)$ is a subquotient of the localization
$\cD_{\alpha} L(s_{\alpha} \cdot \lambda)$ if $\lambda \geq  s_{\alpha}\cdot \lambda$.

\subsubsection{} Assume finally that $\lambda$ is nonintegral.
In this case one easily checks  that if $u$ is a vector in $\cD_{\alpha} L(\lambda)$
such that $e_{\alpha} u = 0$, then $e_{\alpha} (f_{\alpha}^{(\lambda, \alpha)+1} \cdot u) = 0$ in $\cD_{\alpha}^{(\lambda, \alpha)}L(\lambda)$.
From here, arguing as above, we find that $\cD_{\alpha}^{(\lambda, \alpha)} L(\lambda)$
has a subquotient isomorphic to $L(s_{\alpha}\lambda)$ if $(\lambda, \ol{\alpha}) \neq 0$
and has a subquotient isomorphic to $L(s_{\alpha} \cdot \lambda)$ if $(\lambda, \ol{\alpha}) = 0$.

\subsection{}
The following statement will be useful in Section~\ref{stepthmq}.

\begin{lem}{q3-loc}
The ${\mathfrak q}(3)$--module ${\cD}_{\alpha_1} L(s_1 \cdot 0)$ contains a subquotient isomorphic to
$L(s_2 \cdot 0)$.
\end{lem}
\begin{proof}
Let $N(0)$ be the Weyl module of the highest weight zero and $N''$ be its maximal submodule
with the property $N''_0=N''_{-\alpha_1}=0$. The quotient $N':=N(0)/N''$ has length two: its proper
submodule is $L(s_1\cdot 0)=L(-\alpha_1)$ with the quotient isomorphic to the trivial representation $L(0)$.
One has $\dim N'_0=\dim N(0)_0=1$; let $v_0$ be a non-zero vector in $N'_0$.
Let $f_1,f_2 $ (resp., $F_1,F_2$) be the standard even (resp., odd) generators of $\fn^-$ of weights
$-\alpha_1,-\alpha_2$ respectively. The weight space $N'_{-\alpha_1}$ is spanned by the vectors
$v_1:=f_1v_0$ and $F_1v_0$; these are highest weight vectors of $L(s_1\cdot 0)$.  Set
$$u:=(F_2f_1-f_2F_1)v_1.$$
One readily sees that $\fn^+_0u=0$. Let $E_1, E_2$ be the standard  odd generators of $\fn^+$ of weights
$\alpha_1,\alpha_2$ respectively. One has
 $E_1u=0, E_2u=f_1v_1$. In particular, $u\not=0$.

Since $v_1=f_1v_0$ one has
$${\cD}_{\alpha_1} N'={\cD}_{\alpha_1} L(s_1 \cdot 0).$$

Note that $u$ has weight $(s_1s_2)\cdot 0$ so $f_1^{-2}u$ has weight $s_2\cdot 0=-\alpha_2$.
We will show that $\fn^+(f_1^2u)\subset N'$ whereas $f_1^{-2}u\not\in N'$.
This means that the submodule of $({\cD}_{\alpha_1} N')/N'$ generated
by the image of vector $f_1^{-2}u$ is a quotient of $M(s_2 \cdot 0)$ and
this implies the statement.

Since $N'/L(-\alpha_1)\cong L(0)$, one has  $N'_{-\alpha_2}=0$.
In particular, the vector $f_1^{-2}u$ does not lie in $N'$ since it has weight $-\alpha_2$.
Since $\fn^+_0u=0$ and $u$ has weight  $(s_1s_2)\cdot 0$, one has
$\fn^+_0(f_1^{-2}u)=0$. In view of \S \ref{gl2-q2-int} we have $E_1(f_1^{-2}u)=0$.
One has
$$E_2(f_1^{-2}u)=f_1^{-2}(E_2u)=f_1^{-2}f_1v_1=v_0\in N'.$$
Hence $\fn^+(f_1^{-2}u)\in N'$ and $f_1^{-2}u\not\in N'$.
The assertion follows.
\end{proof}

\subsubsection{Remark}
Example~\ref{fam_bounded} provides an explicit realization of $L(s_1\cdot 0)$ and $N'$
(see~\Lem{q3-loc} for notation).
Retain notation of Example~\ref{fam_bounded} and
consider $\C$ as a submodule of ${\mathcal F}_0$ in the natural way. By abuse of notation we will denote all elements in ${\mathcal F}_0/\mathbb{C}$ and in ${\mathcal F}_0$ by the same letters. Set
$$v_0:=x_1^{-1}\xi_1,\ \ v_1:=f_1v=x_1^{-1}\xi_2-x_1^{-2}x_2\xi_1$$
and denote by $L, L'$ the submodules of $\overline{\mathcal F}/\mathbb{C}$ generated by
$v_0,v_1$ respectively. It turns out that $L\cong N'$ and $L'\cong L(s_1\cdot 0)$.

\section{Integral bounded $\ggl_n$--weights in terms of Weyl group orbits}
The description of all $\ggl_n$-bounded weights  provided by O. Mathieu in \cite{M} involves the notion of coherent family. In this section we will give a  description of the  integral bounded weights in terms of Weyl group orbits; in \S  \ref{descr_bounded} we obtain a similar description for
$\gq(n)$.

\subsection{Definitions and $W\cdot$-action}
Recall that $\lambda$ is {\it integral}
 if and only if $s_i\cdot\lambda$ and $\lambda$ are comparable
(i.e., $s_i\cdot\lambda\leq \lambda$ or $s_i\cdot\lambda\geq \lambda$) for each $i$. We note that this integrality condition is different from the condition $\lambda \in {\mathbb Z}^n$.
Also,  $\lambda$ is {\it regular} if $\Stab_{W \cdot}\lambda=\{\Id\}$ and {\it singular} otherwise.
Clearly, if $\lambda$ is integral (resp., regular, singular), then all weights in the orbit $W\cdot\lambda$
are integral (resp., regular, singular).

Henceforth $\lambda$ is called {\it $W$--maximal}
if $s_{\alpha} \cdot \lambda\not\geq\lambda$ for each root $\alpha$.
The stabilizer of a $W$--maximal weight is generated
by simple reflections: if $\lambda$ is $W$--maximal, then
$\Stab_{W\cdot}\lambda=\langle s_i: s_i\cdot\lambda=\lambda\rangle$.
In particular, a $W$--maximal weight is regular if and only if $s_i\cdot \lambda\not=\lambda$ for
$i=1,\ldots,n-1$.

\subsubsection{}\label{dotiii}
We will use the following properties:

(i) $f_{\alpha}$ acts injectively on $\dot{L}(\lambda)$ if and only if $s_{\alpha}\cdot\lambda\not<\lambda$;

(ii) if $s_{\alpha}\cdot\lambda>\lambda$, then $\dot{L}(s_{\alpha}\cdot\lambda)$ is a subquotient
of the (non-twisted) localized module $\cD_{\alpha}\dot{L}(\lambda)$ (the localization functor $\cD_{\alpha}$ is defined in \S \ref{loc-functor});

(iii) the module $\dot{L}(\lambda)$ is finite-dimensional if and only if $\lambda$ is a regular integral
$W$--maximal weight:
$$\dim \dot{L}(\lambda)<\infty\ \Longleftrightarrow\ \forall i\ s_i\cdot\lambda<\lambda;$$

(iv) for each weight $\mu$ there exists a sequence $\mu=\mu_0<\mu_1<\mu_2<\ldots<\mu_s$
such that $\mu_{i+1}=s_{k_i}\cdot\mu_i$ for some $k_i\in\{1,2,\ldots,n-1\}$ and
$\mu_s$ is $W$-maximal. We call such sequence a {\em $W$-increasing string} starting at $\mu$.

\subsection{Integral $\ggl_n$-bounded weights}\label{uniquestringgl}
It is easy to see that if $\dot{L}(\lambda)$ is bounded, then a (nontwisted) localized module
$\cD_{\alpha}\dot{L}(\lambda)$ is also bounded.
Combining (ii) and (iv) we obtain:
if $\mu$ is bounded and $\mu=\mu_0<\mu_1<\mu_2<\ldots<\mu_s$ is an increasing $W$-string, then
each $\mu_i$ is bounded.

As we will show in~\Lem{lemgl3} below, for an integral $\ggl_n$-bounded weight there exists at most one index $i$
such that $f_i$ acts injectively on  $\dot{L}(\lambda)$. Using (i), we obtain:
if $\mu$ is a $\ggl_n$-bounded integral weight,  then

(i) there exists a unique increasing $W$-string $\mu=\mu_0<\mu_1<\mu_2<\ldots<\mu_s$;

(ii) the set $\{i: s_i\cdot\mu_j=\mu_j\}$ is empty for  $j<s$ and has cardinality at most one for $j=s$.

The following proposition shows that an integral weight is bounded if and only if it satisfies (i), (ii).

\subsection{}
\begin{prop}{thmgl}
Let $\gg=\ggl_n$.

(i) Let $\lambda$ be
a regular $W$--maximal integral weight (i.e., $\dot{L}(\lambda)$ is finite-dimensional).
Apart from $\lambda$,
the orbit $W\cdot \lambda$ contains $(n-1)^2$ $\ggl_n$-bounded weights which are of the form
$\prod_{j=i}^k s_j \cdot \lambda$,
$1\leq i, k <n$. The bounded weights of type $i$ in this orbit are $\prod_{j=i}^k s_j
\cdot \lambda$.

(ii) Let $\lambda$ be a singular  $W$--maximal integral weight such that
 $\Stab_{W \cdot}\lambda=\{1,s_m\}$. The orbit
 $W\cdot \lambda$ contains $n-1$ $\ggl_n$-bounded weights which are of the form
$\prod_{j=k}^m s_j\cdot \lambda$,
$1\leq k, m <n$. There is a unique bounded weight of type $k$ in this orbit:
$\prod_{j=k}^m s_j\cdot \lambda $.

(iii) If $\lambda$ is a singular  integral weight such that
$s_i \cdot \lambda=s_j \cdot \lambda=\lambda$ for some $i\not=j$,
then  the orbit $W\cdot \lambda$ does not
contain $\ggl_n$-bounded weights.
\end{prop}

This description easily follows from~\cite{M}. We give a short proof which outlines 
the proof of a similar result for $\gq(n)$; the lemmas appeared in the proof will be used later.

\begin{proof}
Let $\lambda$ be an integral $W$--maximal weight. Then $(\lambda+\rho,\beta)\in\mathbb{Z}_{\geq 0}$
for any $\beta\in\Delta^+$. If $l(s_iw)=l(w)+1$, then $w^{-1}\alpha_i\in\Delta^+$
(see~\cite{Bb}) so
$(w \cdot \lambda+\rho,\alpha_i)=(\lambda+\rho,w^{-1}\alpha_i)\in\mathbb{Z}_{\geq 0}$.
We conclude that
\begin{equation}\label{orderdom}
l(s_iw)=l(w)+1\ \Longrightarrow\ w\cdot\lambda\geq s_iw\cdot\lambda.
\end{equation}

As a result, for each integral $W$-maximal element $\lambda$
and each reduced expression $w=s_{i_k}\ldots s_{i_1}$
one has   a non-decreasing sequence
\begin{equation}\label{eqstringgl}
\lambda'=s_{i_k}\ldots s_{i_1}\cdot\lambda\leq s_{i_{k-1}}\ldots s_{i_1}\cdot\lambda\leq
\ldots\leq s_{i_1}\cdot\lambda\leq \lambda.
\end{equation}

We call such a sequence a {\em non-decreasing $W$-string starting at $\lambda'$}.
By above, the non-decreasing  $W$-strings starting at $\lambda'=w\cdot\lambda$ and ending at $\lambda$
are in one-to-one correspondence with the reduced expressions of the elements
in the set $w\Stab_{W}\lambda$.

Take an integral weight $\mu$ satisfying the conditions (i), (ii)
formulated in Section~\ref{uniquestringgl}. Combining these conditions,
we conclude that
there exists at most two non-decreasing $W$-string starting at $\mu$
and they are of the form
$$\begin{array}{l}
\mu=s_{i_k}\ldots s_{i_1}\cdot\lambda<s_{i_{k-1}}\ldots s_{i_1}\cdot\lambda<
\ldots<s_{i_1}\cdot\lambda<\lambda,\\
\mu=s_{i_k}\ldots s_{i_1}\cdot\lambda<s_{i_{k-1}}\ldots s_{i_1}\cdot\lambda<
\ldots<s_{i_1}\cdot\lambda=\lambda.
\end{array}$$

Recall that $\lambda$ is an integral $W$-maximal weight. By above,
this means that $w$ has a unique reduced expression and, if $s_m\cdot\lambda=\lambda$,
then $ws_m$ has also a unique reduced expression. The set of elements
in $W$ having a unique reduced expression is $B:=\{1, s_is_{i+1}\ldots, s_k, s_is_{i-1}\ldots, s_k\}$.
Thus $w\in B$ for regular $\lambda$,and $w,ws_m\in B$ for singular $\lambda$.
We conclude that $\mu$ appear in the lists (i), (ii) of~\Prop{thmgl}.
In the light of~\ref{uniquestringgl}, all integral bounded weights listed in (i), (ii) of~\Prop{thmgl}.

It remains to verify that all  weights listed in (i), (ii)  are bounded. By Section~\ref{uniquestringgl},
it is enough to check the boundedness of the minimal elements in each string i.e.,
the elements $\prod_{j=1}^k s_j\cdot \lambda, \prod_{j=n-1}^k s_j\cdot \lambda$.
From~(\ref{orderdom})
$$s_i\cdot \prod_{j=1}^k s_j\cdot \lambda\leq \prod_{j=1}^k s_j\cdot \lambda\ \text{ for } i\not=1.$$
If $\lambda$ is regular, then $\Stab_W\lambda=\{1\}$ and the inequalities are strict.
If $\lambda$ is singular and $\Stab_W\lambda=\{1, s_m\}$, then the inequalities are strict for
$k=m$ (because $s_ms_{m-1}\ldots s_1s_is_1s_2\ldots s_m\not=s_m$).
Using~\Lem{lemgl2} we complete the proof.
\end{proof}

\subsection{Lemmas used in the proof of~\Prop{thmgl}}
\subsubsection{}
\begin{lem}{lemgl2}
If $\lambda$ is such that $s_i\cdot \lambda<\lambda$ for $i=1,\ldots,n-2$ (or for $i=2,\ldots,n-1$),
then $\dot{L}(\lambda)$ is bounded.
\end{lem}
\begin{proof}
Let $E$ be the simple
$\ggl_{n-1}\times \ggl_1$-module of highest weight $\lambda$; view $E$ as $\ggl_{n-1}\times \ggl_1+\gn$
module with the trivial action of $e_{\alpha}, \alpha=\vareps_i-\vareps_n$.
The condition $s_i\cdot \lambda<\lambda$ for $i=1,\ldots,n-2$ ensures that $E$ is finite-dimensional and
this implies the boundedness of the induced $\ggl_n$-module
$\Ind_{\ggl_{n-1}\times \ggl_1+\gn}^{\ggl_n} E$. Since $\dot{L}(\lambda)$ is a simple quotient
of this induced module, it is bounded.
\end{proof}

\subsubsection{}
\begin{lem}{lemgl3}
(i) If $\lambda$ is an integral $\ggl_n$-bounded weight,
then there exists at most one simple root $\alpha$ such that $f_{\alpha}$
acts injectively on $\dot{L}(\lambda)$.

(ii) If $\lambda$ is a nonintegral $\ggl_n$-bounded weight,
there exist at most two simple roots $\alpha,\beta$ such that $f_{\alpha}, f_{\beta}$
act injectively  on $\dot{L}(\lambda)$. If $\alpha,\beta$ are such roots (i.e.,
$\alpha,\beta\in\Pi,\ \alpha\not=\beta$ and
$s_{\alpha}\cdot \lambda\not<\lambda,\ s_{\beta}\cdot \lambda\not<\lambda$), then

(a) $(\alpha,\beta)\not=0$ and $s_{\alpha}s_{\beta}s_{\alpha}\cdot \lambda<\lambda$,

(b) $(\lambda,\alpha),(\lambda,\beta)\not\in\mathbb{Z}$.
\end{lem}
\begin{proof}
Let $\alpha,\beta$ be simple roots and $f_{\alpha}, f_{\beta}$
act injectively on $\dot{L}(\lambda)$, i.e.
 $(\lambda,\alpha),(\lambda,\beta)\not\in\mathbb{Z}_{\geq 0}$.
 Assume that $\lambda$ is a $\ggl_n$-bounded weight.

Consider the case $\gg = \ggl_3$ with the simple roots $\alpha,\beta$.
One has $s_{\alpha+\beta}=s_{\alpha}s_{\beta}s_{\alpha}$.
If $s_{\alpha+\beta}\cdot \lambda\not<\lambda$, then
 the Verma module is simple:
$\dot{M}(\lambda)=\dot{L}(\lambda)$ and thus unbounded, a contradiction. Thus
$s_{\alpha+\beta}\cdot \lambda<\lambda$ that is $(\lambda,\alpha+\beta)+1\in\mathbb{Z}_{\geq 0}$.

Now consider the case $\ggl_n$. If $(\alpha,\beta)\not=0$, then considering $\ggl_3$-subalgebra with
the simple roots $\alpha,\beta$ we obtain
$(\lambda+\rho,\alpha),(\lambda+\rho,\beta)\not\in\mathbb{Z}$.

Assume that $(\alpha,\beta)=0$. Let $v$ be a highest weight vector in $\dot{L}(\lambda)$.
Write $\alpha=\vareps_i-\vareps_{i+1}$,
$\beta=\vareps_j-\vareps_{j+1}$ with $i+1<j$. Set $\gamma:=\vareps_i-\vareps_j$.
Note that $f_{\gamma},f_{\beta},e_{\gamma},e_{\beta}$ generate a Lie algebra $\ft$ isomorphic
to $\gs \gl_3$.
Since $f_{\alpha}$ acts injectively,
$f_{\alpha}^kv\not=0$ for any $k$. One has $e_{\gamma}(f_{\alpha}^kv)=e_{\beta}(f_{\alpha}^kv)=0$,
because $\lambda-k\alpha+\gamma\not\leq\lambda, \lambda-k\alpha+\beta\not\leq \lambda$.
Thus $f_{\alpha}^kv$ is a highest weight vector for $\ft$. Since $\dot{L}(\lambda)$ is bounded,
$f_{\alpha}^kv$ generates a bounded $\ft$-module. The highest weight of this   $\ft$-module is $\nu$ such that

$$\begin{array}{l}
(\nu,\beta)=(\lambda-k\alpha,\beta)=(\lambda,\beta),\ \ \
(\nu,\gamma)=(\lambda-k\alpha,\gamma)=(\lambda,\gamma)-k,\\
(\nu,\beta+\gamma)=(\lambda-k\alpha,\beta+\gamma)=(\lambda,\beta+\gamma)-k.
\end{array}$$
By assumptions, $(\lambda,\beta)\not\in\mathbb{Z}_{\geq 0}$. Thus for $k>>0$ one has $(\nu,\beta),\
(\nu,\gamma),\ (\nu,\beta+\gamma)+1\not\in\mathbb{Z}_{\geq 0}$.
Using the $\ggl_3$-considerations above, we conclude that the $\ft$-simple module of  highest weight $\nu$
is not bounded, a contradiction. This completes the proof.
\end{proof}

\section{Bounded modules and  star action}
\subsection{The star action} \label{i-arrow}
The considerations in Section~\ref{q2} lead naturally to the following.

\subsubsection{}
\begin{defn}{}
For $\lambda\in\gh^*$ and $\alpha\in\Pi$  we set
$$s_{\alpha}*\lambda=\left\{
\begin{array}{ll}
s_{\alpha}\lambda &\text{ if } (\lambda,\ol{\alpha})\not=0,\\
s_{\alpha} \cdot \lambda &\text{ if } (\lambda,\ol{\alpha})=0.
\end{array}
\right.$$

For $i=1,\ldots,n-1$ we set $s_i*\lambda:=s_{\alpha_i}*\lambda$.
\end{defn}

We will write $(xy)\cdot\lambda$ and $(xy)*\lambda$ for $x\cdot y\cdot\lambda$ and $x*y*\lambda$ respectively.
For convenience we will write $s_{i_1}...s_{i_k} \cdot \lambda$ and $s_{i_1}...s_{i_k} * \lambda$ for $(s_{i_1}...s_{i_k}) \cdot \lambda$ and $(s_{i_1}...s_{i_k}) * \lambda$, respectively. For example, $s_1s_2 \cdot s_2 s_1 * \lambda = (s_1 s_2) \cdot ((s_2 s_1) * \lambda)$.

\subsubsection{}
Note that $s_{\alpha}*s_{\alpha}*\lambda=\lambda$ and
$s_{\alpha}*s_{\beta}*\lambda=s_{\beta}*s_{\alpha}*\lambda$
if $(\alpha,\beta)=0$. Therefore
the group $\widetilde{W}$ generated by the symbols $s_1,\ldots ,s_{n-1}$
subject to the relations $s_i^2=1,\ s_is_j=s_js_i$ for $i-j>1$
 acts on $\gh^*$ via $*$-action. Note that $\widetilde{W}$ is an infinite Coxeter group. In what follows, each time $w * \lambda$ is written, $w$ is assumed to be an element in $\widetilde{W}$.

\subsubsection{}
We call a weight $\lambda$

{\em $\widetilde{W}$--maximal} if $s_i*\lambda\not>\lambda$ for each $i$,

{\em integral} if $\lambda$ is integral as a $\ggl_n$-weight (note that
this is equivalent to the fact that $\lambda$ and $s_i*\lambda$ are comparable
for each $i$).

We call a $\widetilde{W}$--maximal weight $\lambda$ {\em regular} if  $s_i*\lambda\not=\lambda$ for each $i$
and {\em singular} otherwise. 

It is important to note that the definitions of regular and singular weights above are different from the definitions of regular and singular weights with respect to the star action of $\widetilde{W}$. With our definitions it is easier to formulate the classification theorem for bounded highest weight ${\mathfrak q}(n)$-modules as an analog of the corresponding theorem for $\mathfrak{gl}_n$-modules,  i.e., of~\Prop{thmgl}.

\subsubsection{}\label{stariii}
From \S\ref{def-qn} and \S \ref{q2}   follows that the  $*$-action has the properties
of the $W\cdot$-action listed in Section~\ref{dotiii}:

(i) $f_{\alpha}$ acts injectively on ${L}(\lambda)$ if and only if $s_{\alpha}*\lambda\not<\lambda$;

(ii) if $s_{\alpha}*\lambda>\lambda$, then $\dot{L}(s_{\alpha}\cdot\lambda)$ is a subquotient
of the localized module $\cD_{\alpha}\dot{L}(\lambda)$;

(iii) the module ${L}(\lambda)$ is finite-dimensional if and only if $\lambda$ is a regular integral
$\widetilde{W}$--maximal weight:

$$\dim L(\lambda)<\infty\ \Longleftrightarrow\ \forall i\ \  s_i*\lambda<\lambda.$$

Finally, in~\Prop{propW*max} we will show that

(iv) For each weight $\mu$ there exists a sequence $\mu=\mu_0<\mu_1<\mu_2<\ldots<\mu_s$
such that $\mu_{i+1}=s_{k_i}*\mu_i$ for some $k_i\in\{1,2,\ldots,n-1\}$ and
$\mu_s$ is $\widetilde{W}$--maximal.
We call such sequence a {\em $\widetilde{W}$-increasing string} starting at $\mu$.

\subsubsection{}
As we will see in~\S\ref{exastar} below,  regular $\widetilde{W}$-maximal
weights have different $\widetilde{W}$-orbits (so we do not consider $\widetilde{W}$-{\em regularity});
however, for a maximal regular weight $\lambda$
the weights $(s_j\ldots s_i)*\lambda$ form an increasing string:
$$(s_js_{j-1}\ldots s_i)*\lambda<(s_{j-1}\ldots s_i)*\lambda<\ldots<s_i*\lambda<\lambda,$$
see~\Lem{lems1s2}.

\subsubsection{}
Recall the relation $\succ $ on $\mathbb{C}$ introduced in \S \ref{notation-qn}. 
For $\lambda=\sum_{i=1}^n a_i\vareps_i$ one has
\begin{equation}\label{eqsucc}
\begin{array}{lll}
 \lambda>s_i*\lambda & \Longleftrightarrow\ & a_i\succ a_{i+1};\\
\lambda=s_i*\lambda & \Longleftrightarrow\ & a_i=a_{i+1}\not=0\ \text{ or }
(a_i, a_{i+1})=(-\frac{1}{2},\frac{1}{2});\\
\lambda<s_i*\lambda & \Longleftrightarrow\ & a_{i+1}-a_i\in\mathbb{Z}_{>0}\ \&\
 (a_i,a_{i+1})\not=(-\frac{1}{2},\frac{1}{2})
\end{array}
\end{equation}

In particular, $\lambda$ is a $\widetilde{W}$--maximal weight if for each $i=1,\ldots,n-1$ either
$a_i+a_{i+1}\not=0$ and  $a_i-a_{i+1}\not\in\mathbb{Z}_{<0}$ or $a_i+a_{i+1}=0$ and
$2a_i+1\not\in\mathbb{Z}_{<0}$.

\subsection{Examples}\label{exastar} In contrast to the usual and dot actions, the $*$-action does not induce an action of $W$,
see the examples below.

Fix $\gg = \gq (3)$. Using~\Lem{sectionq3} below, one can show
that there are $6$ types of integral $\widetilde{W}$-orbits (up to the permutation
of $s_1$ and $s_2$, which corresponds to the action of
an automorphism $\iota$, see Section~\ref{iota}).

1) The weights $(-\frac{1}{2},\frac{1}{2},\frac{1}{2}),(-\frac{1}{2},-\frac{1}{2},\frac{1}{2}),$ and $(a,a,a)$ with $a\not=0$ are singular $\widetilde{W}$--maximal weights with the orbits containing only one element.

2) The weight $\lambda=(a,b,c)$ with $a-b, b-c\in\mathbb{Z}_{>0},\  a+b\not=0,  b+c\not=0$ or with $a+b=0, 
2a\in\mathbb{Z}_{\geq 0}, b-c-1\in\mathbb{Z}_{>0}$
is a regular $\widetilde{W}$--maximal weight. Its $\widetilde{W}$-orbit  takes the form

$$
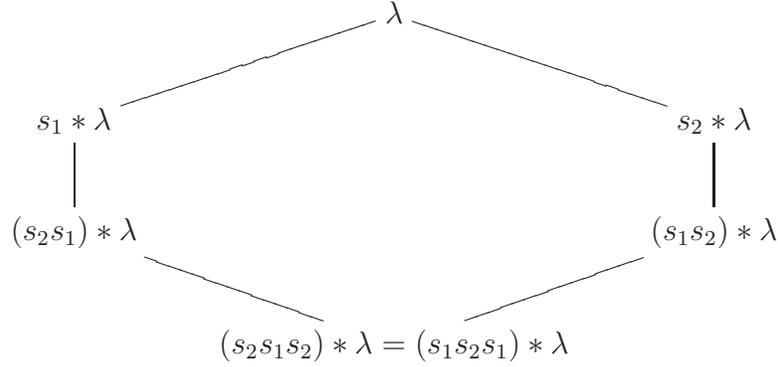
\xymatrix{
& \lambda \ar@{-}[dr] \ar@{-}[dl]  & \\
s_1*\lambda  \ar@{-}[d]    & & s_2*\lambda \ar@{-}[d] \\
(s_2s_1)*\lambda  \ar@{-}[dr] & & (s_1s_2)*\lambda \ar@{-}[dl] \\
& (s_2s_1s_2)*\lambda=(s_1s_2s_1)*\lambda &}$$

The edges of the diagrams correspond to simple reflections $s_1,s_2$ and the upper vertex  in a given edge
is bigger with respect to the partial order.  
The increasing strings are represented by the paths going in upward direction, for instance
$s_1s_2s_1*\lambda<s_2s_1*\lambda<s_1*\lambda<\lambda$.

3)  The weight $\lambda=(a,-a,-a-1)$ with $2a+1\in\mathbb{Z}_{>0}$ is a regular $\widetilde{W}$--maximal weight.
Its $\widetilde{W}$-orbit  contains $9$ elements and has {\em two} $\widetilde{W}$--maximal weights:
$\lambda$, which is regular, and $\lambda+\alpha_1=(a+1,-a-1,-a-1)$, which is singular.

$$\small
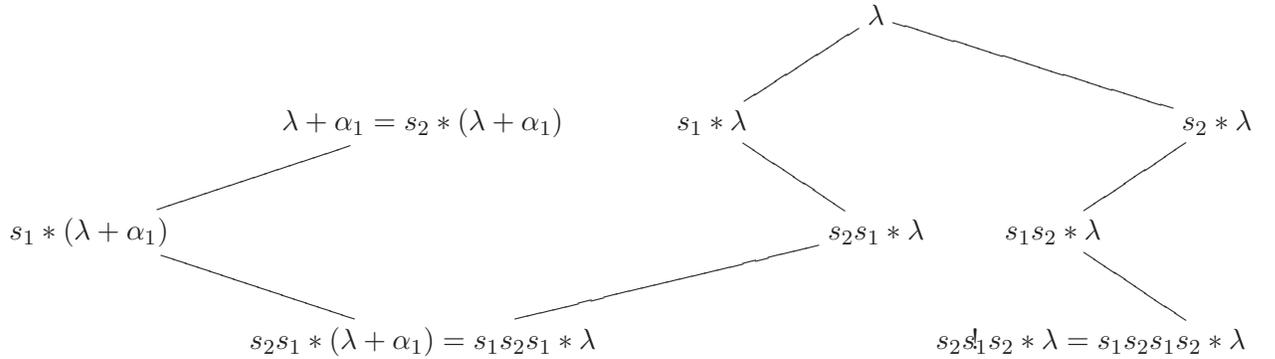
{\xymatrix{
&&&\lambda \ar@{-}[dl]  \ar@{-}[drr] && \\
& \lambda+\alpha_1= s_2 *(\lambda+\alpha_1) \ar@{-}[dl] & s_1*\lambda \ar@{-}[dr] &&&s_2*\lambda \ar@{-}[dl]  \\
s_1 * (\lambda+\alpha_1) \ar@{-}[dr] &&&s_2s_1*\lambda \ar@{-}[dll] &s_1s_2*\lambda\ar@{-}[dr] & \\
& s_2s_1*(\lambda+\alpha_1) = s_1s_2s_1*\lambda &&&&\! \! \! \! \! \! \! \! \! \! \! \! \! \! \! \! \! \! \! \! \! \! \! \! \! \! \! \! \! \! \! \! \! \! \!\! \! \! \! \!\! \! \! \! \! ! \! \! \! \!\! \! \! \! \!  s_2s_1s_2* \lambda = s_1s_2s_1s_2* \lambda }}$$

We see that there are two $\widetilde{W}$-increasing strings starting at $s_1s_2s_1*\lambda$:
one ends at $\lambda$, which is regular, and another one ends at $\lambda+\alpha_1$, which is singular.

4) The weight $0$ is a regular $\widetilde{W}$--maximal weight; its $\widetilde{W}$-orbit
 takes the following form
$$
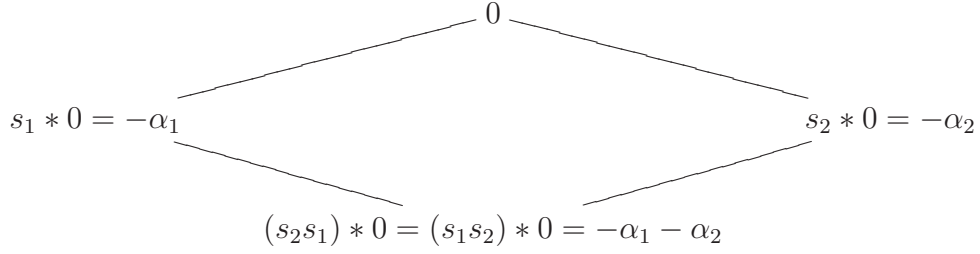
\xymatrix{
& 0 \ar@{-}[dr] \ar@{-}[dl]  & \\
s_1*0=-\alpha_1  \ar@{-}[dr]    & & s_2*0=-\alpha_2 \ar@{-}[dl] \\
& (s_2s_1)*0=(s_1s_2)*0=-\alpha_1-\alpha_2 &}$$

5) The weights $\lambda:=(-\frac{1}{2},\frac{1}{2},-\frac{1}{2}),\
 \lambda+\alpha_1=(\frac{1}{2},-\frac{1}{2},-\frac{1}{2})$
are singular $\widetilde{W}$--maximal weights, lying in the same $\widetilde{W}$-orbit.
 The $\widetilde{W}$-orbit contains
$5$ elements and is of the following form:

$$
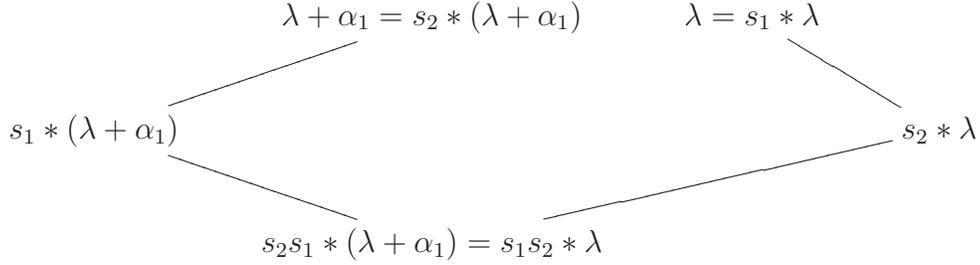
\xymatrix{
& \lambda+\alpha_1=s_2* (\lambda+\alpha_1) \ar@{-}[dl]  &  \lambda = s_1 * \lambda \ar@{-}[dr]  & \\
s_1* (\lambda+\alpha_1)  \ar@{-}[dr]   & & &  s_2* \lambda \ar@{-}[dll] \\
&s_2s_1*(\lambda+\alpha_1) = s_1s_2 * \lambda &&}$$

6) The weights $\lambda:=(-\frac{1}{2},\frac{1}{2},c)$ with $-\frac{1}{2}-c\in\mathbb{Z}_{>0}$
and $\lambda:=(a,a,b)$ with $a,a+b\not=0,\ a-b\in\mathbb{Z}_{>0}$ are singular
$\widetilde{W}$--maximal weights; their $\widetilde{W}$-orbit takes the following form
$$
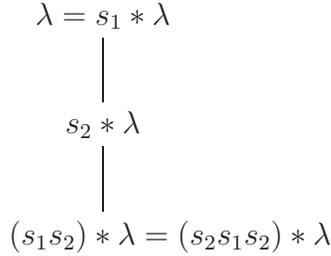
\xymatrix{
& \lambda=s_1*\lambda \ar@{-}[d]   & \\
& s_2*\lambda  \ar@{-}[d] & \\
& \ \ \ \ \ \ \ \ \ \ \ \ \ (s_1s_2)*\lambda=(s_2s_1s_2)*\lambda& }$$

\subsubsection{Remark}
Note that ${L}(\lambda)$ is finite-dimensional if and only if $\lambda$ is represented by a top vertex which belongs 
to $n-1$ edges (where "top" means that there is no edge ascending from this vertex).

As we will show in~\Thm{thmq0},
a integral weight $\mu$ is bounded if and only if there exists a unique ascending path going from $\mu$ and
that each vertex in this path, except the top one, belongs to $n-1$
edges and the top one belongs to at least $n-2$ edges.

\subsubsection{Remark}
Using the above classification of the orbits we obtain $(s_1s_2)^{180}*\lambda=\lambda$
for each weight $\lambda$. Thus $\widetilde{W}$ can be substituted by the group generated
by $1_,\ldots,s_{n-1}$ with the relations $s_i^2=1, (s_is_j)^2=1$ for $|i-j|>1$ and
$(s_is_{i+1})^{180}=1$.

\subsection{}
Later we will need the following lemma.

\begin{lem}{sectionq3}
For $n=3$ the maximal length of a $\widetilde{W}$-increasing  string is $4$ and the maximal element
in a string of length $4$ is regular.
\end{lem}
\begin{proof}
The assertion is equivalent to the following claim:
there is no $\lambda$ satisfying
$$s_1*\lambda<\lambda<s_2*\lambda<s_1s_2*\lambda\leq s_2s_1s_2*\lambda.$$

Assume that $\lambda=(a,b,c)$ satisfies the above inequalities.
By~(\ref{eqsucc}) one has $a\succ b$.

Consider the case $(\lambda,\ol{\alpha_2})\not=0$. Then $s_2*\lambda=(a,c,b)$ and
$s_1s_2*\lambda$ is either $(c,a,b)$ or $(c-1,a+1,b)$.
Since $a\succ b$ one has $a+1>b$ and
so in both cases $s_1s_2*\lambda>s_2s_1s_2*\lambda$,  a contradiction.

Consider the remaining case $(\lambda,\ol{\alpha_2})=0$ that is $b=-c$ and so $a\succ -c$.
Then $s_2*\lambda=(a,c-1,1-c)$. The inequality $s_2*\lambda>\lambda$ gives
$c\geq 1$ and thus the inequality
$a\succ -c$ gives $a+c\in\mathbb{Z}_{>0}$. If $(s_2*\lambda,\ol{\alpha_1})=0$, then
$s_1 s_2*\lambda=(c-2,2-c,1-c)$ so $s_1s_2*\lambda>s_2s_1s_2*\lambda$, a contradiction.
If $(s_2*\lambda,\ol{\alpha_1})\not=0$, then $a\not=1-c$ so $a+c \in\mathbb{Z}_{>1}$. One has
$s_1 s_2*\lambda=(c-1,a,1-c)$ and $a+c \in\mathbb{Z}_{>1}$ gives
$s_1 s_2*\lambda>s_2s_1s_2*\lambda$, a contradiction.
\end{proof}

\subsection{Remark} \label{sing_arbitrary}
In contrast to the Lie algebra case, the notions of ``regularity'' and ``singularity''
are not well-defined for arbitrary integral weight, since the same $\widetilde{W}$-orbit
can contain a regular and a singular  $\widetilde{W}$-maximal weights, see the above example
 for $\gq(3)$ (type 3)).

However the notions of ``regularity'' and ``singularity'' can be defined for the bounded weights.
 Indeed, as we will see in~\Cor{corsal} (iii) below, similarly to the $\ggl_n$-case,
each bounded weight $\lambda'$ lies in at most two $\widetilde{W}$-strings, and the maximal element
in these strings is the same. Thus we can say that a bounded weight $\lambda'$
is regular (resp., singular) if $\lambda$ is regular (resp., singular).

\subsection{}
\begin{prop}{propW*max}
For each $\lambda$ the length of all increasing chains
$$\lambda<s_{i_1}*\lambda<s_{i_2}s_{i_1}*\lambda<
\ldots<s_{i_k}s_{i_{k-1}}\ldots s_{i_1}*\lambda$$
is uniformly bounded. In particular, $W * \lambda$ contains a maximal element $w*\lambda$
such that $\lambda \leq w* \lambda$.
\end{prop}
\begin{proof}
Fix an increasing chain
$$\lambda=\lambda_0<\lambda_1<\ldots <\ \lambda_m=s_{i_m}*\lambda_{m-1}.$$
For each $k=0,\ldots,m$ let $\sneg_k$ be the sum of the negative coordinates of $\lambda_k$:
$\sneg_k:=0$ if $(\lambda_k,\vareps_i)\geq 0$ for every $i=1,\ldots,n-1$ and
$$\sneg_k:=\sum_{i: (\lambda_k,\vareps_i)<0}(\lambda_k,\vareps_i)$$
otherwise.

Call the index $k$ to be dot-type if
$\lambda_k=s_{i_k} \cdot \lambda_{k-1}$ (otherwise $\lambda_k=s_{i_k}\lambda_{k-1}$).
If $k$ is not a dot-type, $s_{i_k}$ acts as a permutation on the coordinates
$(\lambda,\vareps_i)$, so $\sneg_k=\sneg_{k-1}$.
Let $k$ be a dot-type. Then $(\lambda_{k-1},\vareps_{i_k-1})=-a, (\lambda_{k-1},\vareps_{i_k})=a$
for $a\geq 1$ and $(\lambda_{k},\vareps_{i_k-1})=a-1, (\lambda_{k},\vareps_{i_k})=-a+1$.
Thus $\sneg_k=\sneg_{k-1}+1$ if $k$ is a dot-type. Moreover,
if $\sneg_k=0$, then $i$ is not a dot-type for each $i\geq k$.
We conclude that the number of dot-type $k$'s is at most $- \sneg_0$.
Let $N$ be the length of the longest element in the Weyl group $S_n$, i.e. $N=\frac{n(n-1)}{2}$.
There are not  more than $N$ consecutive indices which are not of dot-type: for any $k$
there exists $i$ such that $k\leq i\leq k+N$ and $i$ is a dot-type. We conclude
that the length of any increasing chain is at most $- (N+1)\sneg_0$.
\end{proof}

\subsection{Automorphism $\iota$}
\label{iota}
The diagram automorphism of $\ggl_n$ gives rise to an automorphism $\iota$
of $\gq(n)$. This automorphism stabilizes the Cartan algebra, and the subalgebras $\fn$ and $\fn_-$.
The induced action on $\fh^*$ is an involution given by
$$\iota(\vareps_i):=-\vareps_{n-i}$$
so $\iota(\alpha_i)=\alpha_{n-i},\ \ \iota(\ol{\alpha}_i)=-\ol{\alpha}_{n-i}$.
Note that
$$\iota(s_k*\lambda)=s_{n-k}*\iota(\lambda).$$
In particular, $\iota$ preserves integrality and $\widetilde{W}$-maximality of weights.

The automorphism $\iota$ induces a twisted action on modules and
$L(\lambda)^{\iota}=L(\iota(\lambda))$. Since $\iota$ stabilizes $\fh$, $\lambda$ is bounded
if and only if $\iota(\lambda)$ is bounded. Summarizing, we obtain
$$\prod_{j=i}^k s_j*\lambda \text{ is bounded }\ \Longleftrightarrow\
\prod_{j=n-i}^{n-k} s_j*\iota(\lambda)  \text{ is bounded. }$$

\subsection{}
With the aid of the star action we may generalize the results in \ref{q2} to
$\gq(n)$ and include them in the following proposition.

\subsubsection{} \label{sec_arrow}
\begin{prop}{propqloc}  Let $\alpha$ be a simple root and $\lambda$ be a weight for which
$\lambda \not> s_{\alpha}* \lambda$ (equivalently, $\lambda_i \not\succ \lambda_{i+1}$
if $\alpha = \vareps_i - \vareps_{i+1}$). Then  $f_{\alpha}$
acts injectively on $L(\lambda)$ and  $L (s_{\alpha}  * \lambda )$ is a subquotient of
$\cD_{\alpha}^{(\lambda, \alpha)} L(\lambda)$. In particular, if $\lambda$ is $\alpha$--integral
with $\lambda \leq s_{\alpha} * \lambda$, then $L (s_{\alpha} * \lambda )$ is a subquotient of
$\cD_{\alpha} L(\lambda)$.
\end{prop}

In view of the above proposition we will connect two weight $\mu$ and $\lambda$ by
an $i$-arrow $\mu \xrightarrow{i} \lambda$ if $\cD_{\alpha}^{(\lambda, \alpha)} L(\lambda)$
has a subquotient isomorphic to $L(\mu)$ for $\alpha = \vareps_i - \vareps_{i+1}$. In particular,  $\lambda \xrightarrow{i} s_{\alpha}* \lambda$ if $\lambda > s_{\alpha}* \lambda$ and $
\xymatrix{s_1* (0,0,0) \ar@<0.5ex>[r]^2 & s_2* (0,0,0) \ar@<0.5ex>[l]^1}$ (see \S \ref{q3-loc}).

\subsubsection{}
\begin{cor}{corsal} Let $\lambda$ be a bounded weight. Then

(i) $\lambda$ is $\ggl_n$-bounded;

(ii) if $s_i*\lambda > \lambda$, then
$s_i*\lambda$ is also bounded;

(iii) if $\lambda$ is integral, then
$\#\{i: \lambda\leq s_i*\lambda \}\leq 1$;

(iv) if $\lambda$ is nonintegral, then the set
$\{i: \lambda\not> s_i*\lambda \}$  coincides with the set
$\{i: (\lambda,\vareps_i-\vareps_{i+1})\not\in\mathbb{Z}\}$ and equals either $\{j,j+1\}$,
or $\{1\}$, or  $\{n-1\}$.

\end{cor}
\begin{proof}
(i) follows from the fact that $\dot{L}(\lambda)$ is a $\ggl_n$-quotient of $L(\lambda)$;
(ii) follows from \S~\ref{stariii} (ii).

For (iii), (iv)  let $\dot{L}(\mu)$ be a $\ggl_n$-submodule of $L(\lambda)$.
Since $L(\lambda)$ is bounded,  $\dot{L}(\mu)$ is $\ggl_n$-bounded.
By Section~\ref{stariii} (i), $\lambda\not>s_i*\lambda$ implies that
$f_i$ acts injectively on $L(\lambda)$ and thus on  $\dot{L}(\mu)$. Now~\Lem{lemgl3} (i)
implies (iii). For (iv) note that
$\lambda-\mu$ is integral. By~\Lem{lemgl3} (ii) we obtain
$$\begin{array}{l}
\{i: f_i \text{ acts inj. on } L(\lambda)\}\subseteq  \{i: f_i \text{ acts inj. on } \dot{L}(\mu)\}
= \{i: (\mu,\vareps_i-\vareps_{i+1})\not\in\mathbb{Z}\}=\\
=
\{i: (\lambda,\vareps_i-\vareps_{i+1})\not\in\mathbb{Z}\}\subseteq \{i: f_i \text{ acts inj. on } L(\lambda)\}.
\end{array}$$
Combining with~\Lem{lemgl3} (ii) we obtain (iv).
\end{proof}

\section{Description of  bounded integral weights for $\gq (n)$} \label{descr_bounded}

As in the $\ggl_n$-case (see Section~\ref{uniquestringgl}),
the integral bounded weights can be described by the following theorem.

\subsection{}
\begin{thm}{thmq0}
An integral weight $\mu$ is bounded if and only if

(i) there exists a unique increasing $\widetilde{W}$-string $\mu=\mu_0<\mu_1<\mu_2<\ldots<\mu_s$;

(ii) the set $\{i: s_i*\mu_j=\mu_j\}$ is empty for  $j<s$ and has cardinality at most one for $j=s$.
\end{thm}

The ``only if'' part follows from~\Cor{corsal} (iii); the ``if'' part follows from~\Thm{thmq} below
(see Section~\ref{proofqq}).

\subsection{Notation}
Let $\lambda$ be a regular $\widetilde{W}$-maximal weight $\lambda$. Set
$$z (\lambda) = \# \{ i \; | \; (\lambda, \vareps_i) = 0\}.$$
Let $f(\lambda):=n$ if $z(\lambda)\leq 1$ and $f(\lambda)$ be the minimal
index $i$ such that $(\lambda,\vareps_{i})=0$ if $z(\lambda)\geq 2$.

Observe that for  a singular $\widetilde{W}$-maximal weight $\lambda$
the equality $(\lambda,\vareps_i)=(\lambda,\varepsilon_{i+2})$
does not force $(\lambda,\vareps_i)=(\lambda,\varepsilon_{i+1})$
(for example, $\lambda=\frac{1}{2}\vareps_1-\frac{1}{2}\vareps_2+
\frac{1}{2}\vareps_3$ is $\widetilde{W}$-maximal). However,
$(\lambda,\vareps_i)=(\lambda,\varepsilon_k)=0$ for $k>i$
forces $(\lambda,\vareps_j)=0$ for each $j$ such that
$i\leq j\leq k$.

Writing $\lambda=(a_1,a_2,\ldots,a_n)$ we have
$$a_1>a_2>\ldots> a_{f(\lambda)}$$
and, if $f(\lambda)\not=n$ (or, equivalently, $z(\lambda)\geq 2$), one has
$$0=a_{f(\lambda)}=a_{f(\lambda)+1}=\ldots=a_{f(\lambda)+z(\lambda)-1}>a_{f(\lambda)+z(\lambda)}>\ldots>a_n.$$

\subsection{}

\begin{defn}{def-int-type}
Let $\lambda$ be an integral bounded weight. If $1 \leq k \leq n-1$, we call $\lambda$ a bounded weight of type $k$ if  $\Delta^{\rm{inj}} L (\lambda) = - \Delta_{\gn_k}$ and $\Delta^{\rm{fin}} L (\lambda) =  \Delta_{\gs_k} \sqcup \Delta_{\gn_k} $ (see \S \ref{gl-n-not}).

\end{defn}

\subsection{}
\begin{thm}{thmq}

(i) The integral bounded weights are of the form $\lambda$ or $\prod_{j=i}^k s_j*\lambda$, 
where  $\lambda$ is a $\widetilde{W}$--maximal integral weight such that
$\#\{i| s_i*\lambda=\lambda\}\leq 1$ and the indices $i,k$
satisfy the following conditions:

(a) no conditions for regular $\lambda$ with
$z(\lambda)\leq 2$: there are $(n-1)^2 + 1$ bounded weights for given $\lambda$;

(b) for  regular $\lambda$ with $z(\lambda)\geq 3$ one has either $i=k$ (that is the weight is $s_k*\lambda$),
or $i<k$ and $k\leq f(\lambda)$ or $k\geq f(\lambda)+z(\lambda)-1$, or $i>k$ and
$k\leq f(\lambda)-1$
or $k\geq f(\lambda)+z(\lambda)-2$): there are $(n-1)(n-z(\lambda)+1) + 1$ bounded weights for such $\lambda$;

(c) for singular $\lambda$ the index  $k$ satisfies $s_k*\lambda=\lambda$ (such $k$ is unique by
the assumption on $\lambda$): there are $n-1$  bounded weights for such $\lambda$.

(ii) If $\lambda$ is a $\widetilde{W}$--maximal integral weight and
$\prod_{j=k}^i s_j* \lambda$ is a bounded weight,
then this weight is of type $k$.  In particular, in each case we have
the same number of bounded weights of each type: $n-1$ in the case (a), $n-z(\lambda)+1$ in the case (b),
and $1$ in the case (c).

\end{thm}

\subsubsection{Examples} \label{q_7}
For $n=3$ the weight $\vareps_1$ is a regular
$\widetilde{W}$-maximal integral weight with $z(\vareps_1)=2$.
The bounded weights apart from $\vareps_1$ are:
$$s_1*\varepsilon_1,\ s_1s_2*\varepsilon_1; s_2*\varepsilon_1,\ s_2s_1*\varepsilon_1.$$
The first two weights are of type $1$, and the last two weights are of type $2$.

For $n=7$, the weight $\lambda=\vareps_1-\vareps_6-2\vareps_7=(1,0,0,0,0,-1,-2)$ is a regular
$\widetilde{W}$-maximal integral weight. One has $z(\lambda)=4, f(\lambda)=2$.
The bounded weights apart from $\lambda$ are given by the following table:

$$\begin{tabular}{|l | l|}
\hline
Type & Bounded weights \\
\hline
$1$ & $s_1*\lambda,\ s_1s_2*\lambda,\ s_1\ldots s_5*\lambda,\ s_1\ldots  s_6*\lambda$ \\
\hline
 $2$ & $s_2*\lambda,\ s_2s_1*\lambda, s_2\ldots  s_5*\lambda,\ s_2\ldots s_6*\lambda$ \\
\hline
$3$ & $s_3*\lambda,\ s_3s_2s_1*\lambda, s_3s_4s_5*\lambda,\ s_3\ldots s_6*\lambda$\\
\hline
$4$ & $s_4*\lambda,\ s_4\ldots s_1*\lambda, s_4s_5*\lambda,\ s_4s_5s_6*\lambda$\\
\hline
$5$ & $s_5*\lambda,\ s_5s_4*\lambda, s_5\ldots s_1*\lambda,\ s_5s_6*\lambda$\\
\hline
$6$ & $s_6*\lambda,\ s_6 \ldots s_1*\lambda, s_6s_5*\lambda,\ s_6s_5s_4*\lambda$\\
\hline
\end{tabular}$$

Recall that
$$s_1s_2*(0,0,0)=s_2s_1*(0,0,0)<s_1*(0,0,0),s_2*(0,0,0).$$
This gives $s_2s_3*\lambda=s_3s_2*\lambda<s_2*\lambda,s_3*\lambda$
and shows that $s_2s_3*\lambda$ is not bounded.

\subsection{Preparation to proof of~\Thm{thmq}}
Our proof of~\Thm{thmq} is more complicated than our proof of~\Prop{thmgl}, since, in contrast to the dot-action, the
 $*$-action is not an action of the Weyl group. Recall that our proof of~\Prop{thmgl} is
based on~\Lem{lemgl2}, \Lem{lemgl3} and the inequality~(\ref{orderdom}).

The $\gq(n)$-version of  \Lem{lemgl2} is  the following.

\subsubsection{}
\begin{lem}{lemqn2}
If $\lambda$ is such that $s_i*\lambda<\lambda$ for $i=1,\ldots,n-2$ (or for $i=2,\ldots,n-1$),
then $L(\lambda)$ is bounded.
\end{lem}
\begin{proof}
Like in ~\Lem{lemgl2}, the assertion follows from the fact that for $\gq(n-1)\times \gq(1)$ with the roots
$\alpha_1,\ldots,\alpha_{n-2}$, the simple highest module of weight $\lambda$ is finite-dimensional
if $s_i*\lambda<\lambda$ for $i=1,\ldots,n-2$, see Section~\ref{stariii} (iii).
\end{proof}

\subsection{Proof of ``only if'' part in~\Thm{thmq} (i)}\label{qonlyif}
Let $\lambda'$ be a bounded weight. Consider
a non-decreasing sequence of the form
\begin{equation}\label{step2}
\lambda'=s_{i_1}\ldots s_{i_r}*\lambda< s_{i_2}\ldots s_{i_r}*\lambda<
\ldots <s_{i_r}*\lambda\leq\lambda,
\end{equation}
where $\lambda$ is a $\widetilde{W}$-maximal weight.
By~\Prop{propW*max}, such a sequence exists.
Let us show that $w=s_{i_1}\ldots s_{i_r}$ is of the form described in~\Thm{thmq}.

Assume that $|i_1-i_2|>1$, that is $(\alpha_{i_1},\alpha_{i_2})=0$. For $j=1,2$ let $x_j$
be such that $s_{i_j}*\lambda=\lambda+x_j\alpha_{i_j}$. Then
$$s_{i_2}s_{i_1}*\lambda=\lambda'+x_1\alpha_{i_1}+x_2\alpha_{i_2}.$$
By~(\ref{step2}),
$\lambda'\leq s_{i_2}s_{i_1}\lambda'$ and thus $x_1,x_2\in\mathbb{Z}_{\geq 0}$
that is $\lambda'\leq s_{i_1j}\lambda'$ for $j=1,2$. This contradicts to~\Cor{corsal} (iii).
We conclude that $|i_1-i_2|=1$.

Assume that $i_3=i_1$. By~\Cor{corsal} (iii), $s_{i_2}\lambda'<\lambda'$ so we obtain an increasing sequence
$$s_{i_2}*\lambda'<\lambda'<s_{i_1}*\lambda'<s_{i_2}s_{i_1}*\lambda'<s_{i_1}s_{i_2}s_{i_1}*\lambda'$$
and $|i_1-i_2|=1$. This contradicts to~\Lem{sectionq3}. Hence $i_3\not=i_1$.

By~\Cor{corsal} (ii), all weights in the sequence~(\ref{step2}) are bounded.
Using the above results for the bounded weight $s_{i_j}\ldots s_{i_m}*\lambda$
we obtain $|i_j-i_{j+1}|=1, i_{j+2}\not=i_j$ for each $j$. Thus
$w$ is either of the form $s_is_{i+1}\ldots s_k$ or of the form $s_ms_{m-1}\ldots s_j$.

If $\lambda$ is regular and $z(\lambda)\leq 2$, we obtain that $w*\lambda$ is a weight listed in (a).

Consider the case (b): $\lambda$ is regular and $z(\lambda)\geq 3$.
Let $k$ be such that $(\lambda,\vareps_i)=0$ for $i=k-1,k,k+1$.
By \S\ref{exastar},
$$s_ks_{k-1}*\lambda=s_{k-1}s_k*\lambda<s_k*\lambda,s_{k-1}*\lambda$$
so, by~\Cor{corsal} (iii),
$s_ks_{k-1}*\lambda$ is not bounded. Therefore
$w\not=s_is_{i+1}\ldots s_{k}$ for $i\leq k-1$ and $w\not=s_ms_{m-1}\ldots s_{k-1}$ for $m\geq k$.
Thus $w$ satisfies the conditions  listed in (b).

It remains to consider the case when  $\lambda$ is singular.  By~\Cor{corsal} (iii), the boundedness
of $\lambda$ implies that the cardinality of $\{i|\ s_i*\lambda=\lambda\}$ is at most $1$.
Since $\lambda$ is singular, the cardinality is non-zero, so there exists a unique
index $m$ such that $s_m*\lambda=\lambda$. If $i_r\not=m$, we can extend the sequence~(\ref{step2})
to the sequence
$$w*\lambda=s_{i_1}\ldots s_{i_r}s_m*\lambda< s_{i_2}\ldots s_{i_r}s_m*\lambda<
\ldots <s_{i_r}s_m*\lambda<s_m*\lambda=\lambda$$
which consists of the same weights. By the above, the boundedness of $w*\lambda$ forces
$s_{i_1}\ldots s_{i_r}s_m=s_is_{i+1}\ldots s_m$ or $s_{i_1}\ldots s_{i_r}s_m=s_js_{j-1}\ldots s_m$.
Thus $w*\lambda$ is  a weight listed in (c).

\subsection{Boundedness of the weights listed in~\Thm{thmq}}\label{stepthmq}
By~\ref{sing_arbitrary}, the inequality~(\ref{orderdom}) does not hold for the $*$-action.
We will use the following weaker inequalities, which are proven in Appendix.

\begin{equation}\label{eqs1s2}
\begin{array}{l}
s_j*\lambda<\lambda \text{ for } j=i,i+1,\ldots,k\ \Longrightarrow\
s_is_{i+1}\cdots s_k*\lambda<s_{i+1}\cdots s_k*\lambda.
\end{array}
\end{equation}

\subsubsection{}\label{step1}
Retain notation of~\Thm{thmq} and recall that $f(\lambda)=n$ if $z(\lambda)\geq 1$.

We claim  that the boundedness of the  weights
$\prod_{j=i}^k s_j*\lambda$, where $\lambda$ is a $\widetilde{W}$--maximal integral weight
\begin{equation}\label{choicek}
\left\{\begin{array}{ll}
k\leq f(\lambda) \text{ or }k\geq f(\lambda)+z(\lambda)-1 & \text{ for the cases (a), (b);}\\
k\text{ s.t. } s_k*\lambda=\lambda &  \text{ for the case (c)}
\end{array}\right.\end{equation}
 implies the boundedness of all weights listed in~\Thm{thmq}.

Indeed, combining~\Cor{corsal} (ii) and~(\ref{eqs1s2}), we obtain the boundedness of the weights
of the form $\prod_{j=i}^k s_j*\lambda$ for $i\leq k$ and $k$ as above. Using the automorphism
$\iota$ (see Section~\ref{iota}) we deduce the boundedness of these weights for $i>k$.
 The remaining weights are $s_l*\lambda$
 for $f(\lambda)<l<f(\lambda)+z(\lambda)-1$ in the case (b). The boundedness of these
 weights follows from~\Lem{q3-loc} and
 the boundedness of the weight  $s_{f(\lambda)}*\lambda$. This establishes the claim.

It remains to check the boundedness of the weights $\prod_{j=i}^k s_j*\lambda$ for all $k$ listed in~(\ref{choicek}).
By~\Lem{lemqn2}, the weight $(b_1,\ldots, b_n)$ is bounded if $b_2\succ b_3\succ\ldots \succ b_n$.
 Thus it is enough to verify that for all $k$ listed in~(\ref{choicek}) one has
 \begin{equation}\label{eqstep1}
 b_2\succ b_3\succ\ldots \succ b_n,\ \text{ where }
 (b_1,\ldots,b_n):=\prod_{j=i}^k s_j*\lambda.
 \end{equation}

\subsubsection{Cases (a), (b)}\label{casei}
Let $\lambda$ be regular, that is $\lambda=(a_1,\ldots,a_n)$ with
$$a_1\succ a_2\succ\ldots \succ a_n.$$
We fix $k$ as in~(\ref{choicek}) and define
 $(b_1,\ldots,b_n)$ by the formula~(\ref{eqstep1}).
Note that $b_j:=a_j$ for $j>k+1$.

Consider the case when $k\leq f(\lambda)$. Then $a_1>a_2>\ldots >a_k$.
Using the transitivity of the relation $\succ$
and~\Lem{eqqq}, we obtain
$$ b_2>b_3>\ldots>b_{k+1}\succ b_{k+2}\succ \ldots\succ b_n$$
 and this establishes~(\ref{eqstep1}) for $k\leq f(\lambda)$.

The remaining case is $z(\lambda)\geq 2$ and $k\geq j+z(\lambda)-1$.
Set $j:=f(\lambda)$. One has
$$\begin{array}{l}
a_1>a_2>\ldots>a_{j}=0=a_{j+1}=\ldots= a_{j+z(\lambda)-1}>a_{j+z(\lambda)}>\ldots>a_n,\\
s_js_{j+1}\ldots s_k*\lambda=(a_1,a_2,\ldots,a_{j-1},a_{k+1},a_{j},a_{j+1},\ldots, a_{k},a_{k+2},\ldots, a_n)
\end{array}$$
and so
$$(b_1,b_2,\ldots, b_n)=(b_1,\ldots, b_j,a_{j},a_{j+1}\ldots, a_{k},a_{k+2},\ldots, a_n),$$
where
$$(b_1,b_2,\ldots,b_j):=s_1\ldots s_{j-1}*(a_1,a_2,\ldots,a_{j-1},a_{k+1}).$$
Since $a_1>a_2>\ldots a_{j-1}>a_{k+1}$, \Lem{eqqq} gives
$b_2>b_3>\ldots>b_j$ and $b_j\in \{a_{j-1}, a_{j-1}+1\}$. Since $a_{j-1}>a_j$, we get
$b_j>a_j$. This establishes~(\ref{eqstep1}) in this case and
completes the proof of  boundedness of the weights listed in~\Thm{thmq} (a), (b).

\subsubsection{Case (c)}
In this case  $\lambda=(a_1,\ldots,a_n)$ is  either of the form
$$a_1\succ a_2\ldots \succ a_m=a_{m+1}\succ a_{m+2}\succ\ldots\succ a_n,\ a_m\not=0$$
or of the form

$$a_1\succ a_2\ldots \succ a_m=-\frac{1}{2},\ \frac{1}{2}=a_{m+1}\succ a_{m+2}\succ\ldots\succ a_n.$$
One has
$$s_1\ldots s_m*(a_1,\ldots,a_n)=s_1\ldots s_{m-1}*\lambda=(b_1,b_2,\ldots,b_m,a_{m+1},\ldots, a_n),$$
where
$$(b_1,b_2,\ldots,b_m):=s_1\ldots s_{m-1}*(a_1,\ldots,a_m).$$
Consider $\lambda':=(a_1,a_2,\ldots,a_m)$. By above, $\lambda'$ is a regular integral
$\widetilde{W}$-maximal weight (for $\gq_m$). Since $a_m\not=0$ one has either $f(\lambda')=n$
or $f(\lambda')+z(\lambda')-1<m$. Since~(\ref{eqstep1}) holds for the cases (a), (b), we obtain
$$b_2\succ b_3\succ\ldots\succ b_m.$$
Moreover, $b_m\in\{a_{m-1},a_{m-1}+1\}$ so $b_m\succ a_m$.
This establishes~(\ref{eqstep1}) and completes the proof of  boundedness of the weights listed in~\Thm{thmq} (c).

\subsection{Proof of (ii)}
Recall that  $f_i$ acts injectively on $L(\lambda')$
if and only if $s_i*\lambda'\not<\lambda'$. Let $\lambda$ be a $\widetilde{W}$-maximal weight
such that $\#\{i: s_i*\lambda=\lambda\}\leq 1$.
If $\lambda$ is regular, take an arbitary $k$; if $\lambda$ is singular take
$k$ such that $s_{k+1}*\lambda=\lambda$. Then, by~(\ref{eqs1s2}), for $i<k$ one has
$s_is_{i+1}\ldots s_k*\lambda<s_{i+1}\ldots s_k*\lambda$
so $f_i$ acts injectively on $L(\prod_{j=i}^k s_j*\lambda)$.
Using the automorphism $\iota$, we obtain this for $i>k$. This establishes (ii) and completes the proof.
\qed

\subsection{Proof of~\Thm{thmq0}}\label{proofqq}
It remains to show that the weights satisfying (i), (ii) of~\Thm{thmq0}
are the weights listed in~\Thm{thmq}. Let $\lambda$ be a $\widetilde{W}$-maximal weight
and $w*\lambda$ satisfies the conditions (i), (ii). Write $w=s_{i_1}s_{i_2}\ldots s_{i_r}$
and set $\mu_j:=s_{i_j}\ldots s_{i_{r}}*\lambda$. One has
$\{i:\ s_i*\mu_j\geq \mu_j\}=s_{i_j}$ for $j=1,\ldots, r$.
In particular, for $j<r$ one has
$$s_{i_{j+1}}*\mu_j<\mu_j<s_{i_j}*\mu_j<s_{i_j}s_{i_{j+1}}*\mu_j$$
which  implies  $|i_j-i_{j+1}|=1$  (see \S\ref{qonlyif}).
Assume that $i_{j+1}=i_{j-1}$. Then
$$s_{i_j}s_{i_{j-1}}s_{i_j}s_{i_{j-1}}*\mu_{j-1}<s_{i_{j-1}}s_{i_j}s_{i_{j-1}}*\mu_{j-1}<
s_{i_j}s_{i_{j-1}}*\mu_{j-1}<
s_{i_{j-1}}*\mu_{j-1}<\mu_{j-1}.$$
Since $|i_j-i_{j+1}|=1$, this contradicts to~\Lem{sectionq3}.

We conclude that $w=\prod_{j=i}^k s_j$ for some $i,k$. By (ii), if $\lambda$ is singular, then
$k$ is such that $s_k*\lambda=\lambda$; this coincides the condition (c) in~\Thm{thmq}.
Finally, let us show that the conditions (b) in~\Thm{thmq} hold if
$z(\lambda)\geq 3$. Indeed, assume that $(\lambda,\vareps_j)=0$ for $j=i,i+1,i+2$
and $i_r=i+1, i_{r-1}=i$. Then $\mu_{r-1}=s_{i}s_{i+1}*\lambda$. By Section~\ref{exastar} in this case
$s_is_{i+1}*\lambda=s_{i+1}s_i*\lambda<s_i*\lambda,s_{i+1}*\lambda$ that is
$s_k\mu_{r-1}>\mu_{r-1}$ for $k=i,i+1$, a contradiction. The assertion follows.
\qed

\subsection{Families of integral bounded modules}

\subsubsection{} \label{subsec-fam}
\begin{defn}{reg_fam} Let $\lambda$ be as in~\Thm{thmq} (i)(a) and (i)(b). In particular, $\lambda$ is a $\widetilde{W}$-maximal regular integral weight. If $\prod_{j=i}^k s_j*\lambda$ is a bounded weight,  we call the module $L(\prod_{j=i}^k s_j*\lambda)$  a
{\it regular integral bounded module of type $i$}, and, also, a {\it bounded module of regularity $k$}. If $z(\lambda) \leq 2$,  or $z(\lambda) >2$ and $i \notin [f(\lambda), f(\lambda) + z(\lambda) - 2]$, the set of all  $L(\prod_{j=i}^k s_j*\lambda)$, $1\leq i \leq n-1$, is called {\it the regular integral family of $\lambda$ of regularity $k$}. If $z(\lambda) >  2$ we call the set containing all $L(\prod_{j=i_1}^{f(\lambda)} s_j*\lambda), L(s_{i_2}*\lambda),  L(\prod_{j=f(\lambda)+z(\lambda) -2}^{i_3} s_j*\lambda)$ for $i_1 \leq f(\lambda), f(\lambda) < i_2< f(\lambda) + z(\lambda) - 2, i_3 \geq f(\lambda) + z(\lambda) - 2$ {\it the regular integral family of $\lambda$ of regularity $f(\lambda),...,f(\lambda)+z(\lambda)-2$}.
\end{defn}

Every regular integral family can be represented by a connected graph with vertices the weights of the modules and arrows $\xleftarrow{j}$ and  $\xrightarrow{j}$, see Example \ref{ex_fam}.

\subsubsection{}
\begin{defn}{sing_fam} Let $\lambda$ be as in~\Thm{thmq} (i)(c) and let $k$ be such that
$s_k*\lambda=\lambda$. The module $L(\prod_{j=i}^k s_j*\lambda)$ is  called a
{\it singular  bounded module of type $i$}, and, also, a {\it  bounded module of singularity $k$}. In particular, $L(\lambda)$ is a bounded module of type $k$ and singularity $k$.  The set of
$n-1$ modules $L(\prod_{j=i}^k s_j*\lambda)$, $1 \leq i \leq n-1$,
is called the {\it singular family of $\lambda$ of singularity $k$}.

\end{defn}

\subsubsection{Remarks} (i) In both integral (regular and singular) cases, every  family  has exactly one module of each type. While a singular family of given singularity type can not have any other singularity type, this is not true for the regular integral  families and their regularities (see Example \ref{ex_fam}).
Recall the notation $\lambda \xrightarrow{i} \mu$ introduced in \S \ref{sec_arrow}. The weights of the modules in a regular integral family of $\lambda$ of regularity $k$ for $z(\lambda) \leq 2$ can be described by the connected graph
$$
s_{1}...s_{k}*\lambda \xleftarrow{1}...  \xleftarrow{k-2}s_{k-1} s_k * \lambda  \xleftarrow{k-1}
s_k * \lambda \xrightarrow{k+1} s_{k+1}s_k * \lambda  \xrightarrow{k+2} ... \xrightarrow{n-1} s_{n-1}... s_{k}*\lambda
$$
while those in a singular family of $\lambda$ of singularity $k$  by
$$
s_{1}...s_{k-1}*\lambda \xleftarrow{1}...  \xleftarrow{k-2}s_{k-1}* \lambda  \xleftarrow{k-1}
\lambda \xrightarrow{k+1} s_{k+1}* \lambda  \xrightarrow{k+2} ... \xrightarrow{n-1} s_{n-1}... s_{k+1}*\lambda
$$

Recall that in the former case we also have  $\lambda \xrightarrow{k} s_k * \lambda$, while in the latter case $\lambda = s_k * \lambda$.

(ii) There are more arrows in the above graphs, but we do not need them at the moment.

(iii) The singular bounded modules of type $i$ have the same shadow as the  regular integral bounded modules of type $i$.

\subsubsection{Example} \label{ex_fam}

 We continue the example in \S \ref{q_7} for $\lambda=(1,0,0,0,0,-1,-2)$. Below we describe all (total four) regular integral bounded families of $\lambda$.

$$\begin{tabular}{|l | l|}
\hline
Regularity & Bounded weights in the family of $$ \\
\hline
$1$ & $s_1*\lambda,\ s_2s_1*\lambda,\ s_3s_2s_1*\lambda,\ s_4...s_1*\lambda,\ s_5...s_1*\lambda,\ s_6...s_1*\lambda $ \\
\hline
 $2,3,4$ & $s_2*\lambda,\ s_1s_2*\lambda, s_3*\lambda,\ s_4*\lambda,\ s_5s_4*\lambda,\ s_6s_5s_4*\lambda$ \\
\hline
$5$ & $s_5*\lambda,\ s_4s_5*\lambda,\ s_3s_4s_5*\lambda,\ s_2...s_5*\lambda,\ s_1...s_5*\lambda,\ s_6s_5*\lambda $\\
\hline
$6$ & $s_6*\lambda,\ s_5s_6*\lambda,\ s_4s_5s_6*\lambda,\ s_3...s_6*\lambda,\ s_2...s_6*\lambda,\ s_1...s_6*\lambda $\\
\hline
\end{tabular}$$

The family of $\lambda$ of  regularity $2$, $3$, $4$ can be described by the graph
$$
\xymatrix{s_1 s_2 * \lambda& s_2 * \lambda
\ar@<0ex>[l]^{\; \; \; \;1} \ar@<0.5ex>[r]^3 & s_3 * \lambda \ar@<0.5ex>[l]^2
\ar@<0.5ex>[r]^4 & \ar@<0.5ex>[l]^3 s_4* \lambda \ar@<0.5ex>[l]^2
\ar@<0ex>[r]^{\! \! \! \! \!5} & s_5s_4* \lambda
\ar@<0ex>[r]^{6\;\;} & s_6s_5s_4* \lambda}
$$

\subsection{Nonintegral bounded weights}

\subsubsection{}

\begin{defn}{def-nonint-type}Let $1\leq k \leq n-1$ and $\lambda$ be a nonintegral bounded weight such that
 $\Delta^{\rm{inj}} L (\lambda) = - \Delta_{\gn'_k}$ and
$\Delta^{\rm{fin}} L (\lambda) =  \Delta_{\gs_k'} \sqcup \Delta_{\gn_k'} $ (see \S \ref{gl-n-not}). If $k=1$ (respectively, $1<k<n-1$, $k=n-1$) we call $\lambda$ a {\it bounded weight of type $1$ (respectively, of type $(k,k+1)$, of type $n-1$)}.

\end{defn}

\subsubsection{}
\begin{thm}{thmnint}
(i) A nonintegral weight for $\gq(n)$ is bounded if and only if it is of the form
$\lambda,\ s_m  \ldots s_{1}*\lambda$, $1\leq m \leq n-1$, where
$\lambda$ is a nonintegral weight such that $s_j * \lambda < \lambda$ for $j=2,3,...,n-1$.

(ii) The element $f_{\vareps_s-\vareps_{s+1}}$
acts injectively on $L(s_m s_{m-1} \ldots s_{1}*\lambda)$ if and only if

$\bullet$ $s=m$ for $m=1,n-1$;

$\bullet$ $s\in\{m,m+1\}$ for $2 \leq m \leq n-2$.

\end{thm}

\begin{proof}
Write $\lambda=\sum_{i=1}^n a_i\vareps_i$.
By the assumption, $a_1-a_2\not\in\mathbb{Z}$ and $a_i\succ a_{i+1}$ for $i\not=1$.
Set  $y_0:=a_1$ and introduce $x_i,y_i$ for $i=1,\ldots, n-1$ by the formulas
$$x_i:=(s_i \ldots s_1*\lambda,\vareps_i),\ \  y_i:=(s_i \ldots s_1*\lambda,\vareps_{i+1}),$$
and note that
\begin{equation}\label{nums}
s_k \ldots s_1*\lambda=\sum_{i=1}^k x_i\vareps_i+ y_k\vareps_{k+1}+\sum_{i=k+2}^{n-1} a_i\vareps_i.
\end{equation}

One readily sees that
\begin{equation}\label{xiyi}
(x_i,y_i)= \left\{\begin{array}{lll}
(a_{i+1},y_{i-1} ) & \text{ if } & y_{i-1}+a_{i+1}\not=0,\\
(a_{i+1}-1,y_{i-1}+1) & \text{ if } & y_{i-1}+a_{i+1}=0.
\end{array}\right.
\end{equation}
Assume  that $x_{i-1}\not\succ x_{i}$ for some $i>1$. Then, by~(\ref{xiyi}),
either $(x_{i-1},x_i)=(a_i-1,a_{i+1}-1)$  and $a_i=a_{i+1}=0$ or
$(x_{i-1},x_i)=(a_i-1,a_{i+1})$ and $a_i-a_{i+1}=1$. In both cases~(\ref{xiyi}) implies
$y_{i-2}+a_{i}=0$ and  $y_{i-1}=y_{i-2}+1=1-a_i$. In the first case the first formula gives
$y_{i-2}=a_i=0$ so both $a_1=y_0$ and $a_2$ are integers, a contradiction with $a_1-a_2\not\in\mathbb{Z}$.
In the second case we obtain
$y_{i-1}=1-a_i=-a_{i+1}$ and so, by~(\ref{xiyi}), $x_i=a_{i+1}-1$, a contradiction.

We conclude $x_{i-1}\succ x_{i}$ for each $i>1$. The formula~(\ref{nums}) implies (ii).
By~\Lem{lemqn2}, $\lambda$ is bounded. Combining~\Prop{propqloc} and (ii),
we conclude that the weights listed in (i) are bounded.

For (iii) take a nonintegral bounded weight $\lambda'$. By~\Cor{corsal} (iv) the set
$$S(\lambda'):=\{i:s_i*\lambda'\not<\lambda'\}=\{i|\ (\lambda',\vareps_i-\vareps_{i+1})\not\in\mathbb{Z}\}$$
is either $\{1\}$ or $\{n-1\}$ or of the form
$\{k,k+1\}$.

In the case $S(\lambda')=\{1\}$, $\lambda:=\lambda'$ satisfies the assumption of the theorem.
Consider the case $S(\lambda')\not=\{1\}$;
let $k$ be the minimal element in $S(\lambda')$. Let us show that
$$\lambda:=s_1\ldots s_k*\lambda'$$
satisfies the assumption of the theorem.
We proceed by induction on $k$. Observe that,
 since $\lambda'$ is bounded and $s_k*\lambda'\not<\lambda'$, the weight $s_k*\lambda'$
is bounded (see~\Prop{propqloc}).  We define $S(s_k*\lambda')$ similarly to $S(\lambda')$.

Consider the case $k=1$ that is $S(\lambda')=\{1,2\}$.
Write $\lambda'=\sum_{i=1}^n a_i\vareps_i$.
Since $\lambda'$ is bounded, $\sum_{i=1}^{3} a_i\vareps_i$ is a $\ggl_3$-bounded weight.
In particular, the Verma  $\ggl_3$-module is not simple so
$a_1-a_{3}\in\mathbb{Z}_{\geq 0}$ (because $S(\lambda')=\{1,2\}$ so $a_1-a_2,a_2-a_3$ are not integral).
Therefore $(s_1*\lambda',\vareps_2-\vareps_3)\in\mathbb{Z}$ that is $2\not\in S(s_1*\lambda')$.
For $j>2$ one has $j\not\in S(\lambda')$ and so $j\not\in S(s_1*\lambda')$.
Hence $s_j*(s_1*\lambda')<s_1*\lambda'$ for $j\geq 2$.
We conclude that $S(s_1*\lambda')=\{1\}$ and thus
 $\lambda=s_1*\lambda'$ satisfies the assumption of the theorem.

Consider the case $k>1$.  Since the values
$(\lambda,\vareps_k-\vareps_{k+1}),\ (\lambda,\vareps_{k-1}-\vareps_{k+1})$
 are not integral, the values
$(s_k*\lambda,\vareps_k-\vareps_{k+1}),\ (s_k*\lambda,\vareps_{k-1}-\vareps_{k})$
are not integral as well. Hence $k-1,k\in S(s_k*\lambda')$. Since $s_k*\lambda'$ is bounded,
$S(s_k*\lambda')$ contains at most two elements and thus $S(s_k*\lambda')=\{k-1,k\}$.
By induction, $(s_1\ldots s_{k-1})*(s_k*\lambda')=\lambda$ satisfies the assumption of the theorem.
This completes the proof.
\end{proof}

\subsection{}
\begin{cor}{}
If $\lambda$ is a nonitengral bounded weight of type $1$, then $s_i...s_1*\lambda$ is a  nonitengral bounded weight of type $(i,i+1)$ if $1\leq i \leq n-2$ and a nonitengral bounded weight of type $n-1$ if $i=n-1$.
\end{cor}

\subsection{}
\begin{defn}{nint_fam} Let $\lambda$ be as in Theorem \ref{thmnint} and let $1\leq m \leq n-1$.

(i) The set of $n$ modules
$
\{ L(\lambda), L( s_i  \ldots s_1*\lambda)\; | \; 1 \leq i \leq n-1\}
$
is called a {\it nonintegral family of $\lambda$}.

(ii) The module $L(\lambda)$ is called a {\it nonintegral bounded module of type} $i$ (respectively, $(i,i+1)$), if $\lambda$ is a nonintegral bounded weight of type $i$ (respectively, $(i,i+1)$), see Definition \ref{def-nonint-type}.
\end{defn}

\begin{rem}{} (i) Every nonintegral family  has exactly one module of each type $1, (1,2),...,(n-2,n-1), n-1$. The weights of the modules in such a family can be described by the graph
$$
\xymatrix{ \lambda
 \ar@<0.5ex>[r]^{\! \! \! 1} & s_1 * \lambda  \ar@<0.5ex>[l]^{\! \! \! 1}
\ar@<0.5ex>[r]^{2 \! \! \! \!}& \ar@<0.5ex>[l]^{2\! \! \! \!} ...   \ar@<0.5ex>[r]^{n-1 \hspace{1cm}}
&  \ar@<0.5ex>[l]^{n-1 \hspace{1cm}} s_{n-1}...s_1* \lambda}
$$
where $\lambda$ is of type $1$.

(ii) Theorem \ref{thmnint} can be reformulated in terms of a nonintegral weight of type $n$. Indeed, it is not difficult to prove that $\lambda$ is of type $n$ if and only if $s_{1} s_{2} ...s_{n-1}*\lambda$ is of type $1$. So, alternatively, every nonintegral bounded module of type $m$ is of the form $L(s_m  s_{m+1}...s_{n-1} * \lambda)$ where $\lambda$ is a nonintegral weight with $s_j * \lambda < \lambda$ for $j = 1,2,...,n-2$ (equivalently, $\lambda$ is of type $n$).

\end{rem}

\subsection{}
\begin{cor}{} \label{cor-q-bounded}
Proposition \ref{gl-bounded} remains valid in the case of $\gq(n)$, i.e.
if $\dot{L}(\lambda)$ is replaced by $L(\lambda)$.
\end{cor}

\subsection{Remark}
Recall that $\lambda$ is a $\ggl_n$-bounded weight
if and only if the sequence $(a_1,\ldots ,a_n)$ defined by
$\lambda+\rho=:\sum_{i=1}^n a_i\vareps_i$ has the following property:
for some index $j$ one has $a_{j-1}-a_{j+1}\in\mathbb{Z}_{>0}$ and
$a_i-a_{i+1}\in\mathbb{Z}_{>0}$ for $i\not=j$. It turns out that
the similar description ($\lambda=\sum_{i=1}^n a_i\vareps_i$ is bounded if and only if
for some index $j$ one has $a_{j-1}\succ a_{j+1}$ and
$a_i\succ a_{i+1}\in\mathbb{Z}_{>0}$ for $i\not=j$)
 does not hold for $\gq (n)$.

For example, the weight $(1,-1,1,-1)$
is bounded, however does not satisfy this condition.
The boundedness of $(1,-1,1,-1)$ follows from~\Cor{corsal} (ii),~\Lem{lemqn2} and
the fact that $(1,-1,1,-1)=s_1*(-2,2,1,-1)>(-2,2,1,-1)$.
A nonintegral counterexample is the weight $(a,-a,a)$ for $2a\not\in\mathbb{Z}$:
taking into account that $s_1*(a,-a,a)=(-a-1,a+1,a)$ one obtains
the boundedness of $(a,-a,a)$ from~\Prop{propqloc} and~\Lem{lemqn2}.

On the other hand, the weight  $(a,-a,a-1)$ satisfies
the above condition, but it is not bounded for $2a\not\in\mathbb{Z}$:
indeed, $s_1*(a,-a,a-1)=(-a+1,a-1,a-1)$ and, by~\Prop{propqloc},
the boundedness of $(a,-a,a-1)$ is equivalent to the boundedness of
$(-a+1,a-1,a-1)$ which fails by~\Cor{corsal} (iv).

\section{$\gg \gl_n$-structure of bounded $\gq(n)$-modules} \label{sec_gl_n}
In this section we will study the $\gg \gl_n$-structure of bounded modules. We will prove Propositions \ref{prop-loc-fin} and \ref{prop-link}.

\subsection{Definitions}

\subsubsection{}\label{gl-n-fam} We write $\xymatrix{
 \lambda \ar@{-->}[r]^{i} & \mu}$ if $\dot{L}(\lambda)$ is a $\gg \gl_n$-subquotient of $\cD_{\alpha_i} \dot{L} (\mu)$, where $\alpha_i = \vareps_i - \vareps_{i+1}$. Like in the $\gq(n)$-case, we have three types of $\gg \gl_n$ bounded families (for details see \cite{M}):

 (i) A regular integral $\gg \gl_n$-family of $\lambda$ of regularity $k$
 $$
\xymatrix{s_1...s_k \cdot \lambda
 \ar@{-->}@<0.5ex>[r]^{2 \! \! \! \! \! \!} & ... \ar@{-->}@<0.5ex>[l]^{1 \! \! \!\! \! \!}
\ar@{-->}@<0.5ex>[r]^{\! \! \! \! \!  k }& \ar@{-->}@<0.5ex>[l]^{k-1} s_k \cdot \lambda   \ar@{-->}@<0.5ex>[r]^{k+1}
&  \ar@{-->}@<0.5ex>[l]^{k \! \! \! \! \! } ... \ar@{-->}@<0.5ex>[r]^{ \! \! \! \! \! \! \! \! \! \! n-1  }& \ar@{-->}@<0.5ex>[l]^{ \! \! \! \!\! \! \! \! \!  \! n-2} s_{n-1}...s_k  \cdot \lambda }
$$
where $\lambda$ is a $\gg \gl_n$-dominant integral  weight. In addition, ``outside the family'' we have $\xymatrix{
 \lambda \ar@{-->}[r]^{k} & s_k \cdot \lambda}$.

 (ii) A singular  $\gg \gl_n$-family of $\lambda$ of singularity $k$
 $$
\xymatrix{s_1...s_{k-1} \cdot \lambda
 \ar@{-->}@<0.5ex>[r]^{2 \! \! \! \! \! \!  \! \! \! \! \! \!} & ... \ar@{-->}@<0.5ex>[l]^{1  \! \! \! \! \! \! \! \! \!\! \! \!}
\ar@{-->}@<0.5ex>[r]^{\! \! \! \! \!  k }& \ar@{-->}@<0.5ex>[l]^{k-1} \lambda   \ar@{-->}@<0.5ex>[r]^{k+1}
&  \ar@{-->}@<0.5ex>[l]^{k \! \! \! \! \! } ... \ar@{-->}@<0.5ex>[r]^{ \! \! \! \! \! \! \! \! \! \!  \! \! \! \! \! \! n-1  }& \ar@{-->}@<0.5ex>[l]^{ \! \! \! \!\! \! \! \! \!  \!  \! \! \! \! \! \! n-2} s_{n-1}...s_{k+1}  \cdot \lambda }
$$
where $\lambda$ is a $W$-maximal integral $\gg \gl_n$-bounded weight  with $s_k \cdot \lambda = \lambda$.

 (iii) A nonintegral  $\gg \gl_n$-family of $\lambda$
 $$
\xymatrix{ \lambda
 \ar@{-->}@<0.5ex>[r]^{\! \! \! 1} & s_1 \cdot \lambda  \ar@{-->}@<0.5ex>[l]^{\! \! \! 1}
\ar@{-->}@<0.5ex>[r]^{2 \! \! \! \!}& \ar@{-->}@<0.5ex>[l]^{2\! \! \! \!} ...   \ar@{-->}@<0.5ex>[r]^{n-1 \hspace{1cm}}
&  \ar@{-->}@<0.5ex>[l]^{n-1 \hspace{1cm}} s_{n-1}...s_1 \cdot  \lambda}
$$
 where $\lambda$ is a $\gg \gl_n$-bounded nonitegral weight of type $1$.

\subsubsection{}  \begin{rem}{}
 (i) For every bounded integral weights $\lambda$ and $\mu$ and $i> 1$, $\xymatrix{
\lambda \ar@{-->}[r]^{i} & \mu}$ if and only if $\xymatrix{
\mu \ar@{-->}[r]^{i-1} & \lambda}$.

(ii) For every bounded nonintegral weights $\lambda$ and $\mu$, $\xymatrix{
\lambda \ar@{-->}[r]^{i} & \mu}$ if and only if $\xymatrix{
\mu \ar@{-->}[r]^{i} & \lambda}$.

 \end{rem}

\subsubsection{}\label{notation-ss} For a $\gq(n)$-module (respectively, $\gg \gl_n$-module) $M$ of finite length, by $M_{\rm ss}$ (resp., $M_{\gg \gl {\rm -ss}}$) we denote the direct sum of  all simple subquotients (with multiplicities) of $M$. When the notation $M_{\gg \gl {\rm -ss}}$ is used for
 a $\gq (n)$-module $M$, we consider $M$ as a $\gg \gl_n$-module. For a weight $\lambda$ set
$$
JH (\lambda) := \{ \mu \in \gh_{\bar{0}}^* \; | \; \dot{L}(\mu) \mbox{ is a subquotient of } L(\lambda)\}.
$$

\subsection{}

 \begin{prop}{prop-loc-fin}
 Let $\lambda$ be a bounded integral weight of type $i$. Then $\cD_{\alpha_i} L(\lambda)$ has unique simple subquotients of type $i-1$ (for $i>1$) and $i+1$ (for $i<n-1$).
 \end{prop}
 \begin{proof}
 For a $\gg \gl_n$-module $M$, and $1\leq i \leq n-1$, denote by $M[i]$ the submodule of $M$ consisting of $\gp_i$-locally finite vectors, i.e.
of all $m \in M$ for which for every root $\alpha$ of $\gp_i$ there is positive integer $s$ such that $x^sm =0$ for $x \in \gp_i^{\alpha}$.  We start with the following lemma.

  \begin{lem}{gl-n-loc-fin}
 Let $M$ be a  $\gg \gl_n$-bounded module of finite length all of whose infinite dimensional simple subquotients are highest weight  modules of type $i$, $i>1$. Then $N = (\cD_{\alpha_{i}} M)_{\gg \gl {\rm -ss}}[i-1]$ is a semisimple bounded module  all of whose simple submodules are highest weight  modules of type $i-1$. Furthermore, $(\cD_{\alpha_{i-1}} N)_{\gg \gl {\rm -ss}}[i]  \simeq M_{\gg \gl {\rm -ss}} / M_{\rm f}$, where $M_{\rm f}$ is the direct sum of all finite dimensional subquotients of $M$.
\end{lem}

The proof of the lemma follows from the exactness of the localization functor and the description of the families of integral $\gg \gl_n$-bounded weights in \S \ref{gl-n-fam}.

 Since the two cases in the proposition are proved with the same reasoning, we show the uniqueness of simple subquotients of type $i-1$ only. The statement is equivalent to showing that the $\gq(n)$-module $(\cD_{\alpha_i} L(\lambda))_{\rm ss} [i-1]$ is simple. Let $\dot{L} (\mu) = (\cD_{\alpha_i} \dot{L}(\lambda))_{\gg \gl {\rm -ss}} [i-1]$. By the lemma,  $\dot{L}(\lambda)   \subset (\cD_{\alpha_{i-1}} \dot{L}(\mu))_{\gg \gl-{\rm ss}} [i]$. Let $\lambda'$ be such that $\dot{L} (\mu) \subset L(\lambda')_{\gg \gl-{\rm ss}}$ and $L(\lambda') \subset (\cD_{\alpha_i} L(\lambda))_{\rm ss} [i-1]$. We show that $L(\lambda') = (\cD_{\alpha_i} L(\lambda))_{\rm ss} [i-1]$.

 Let $N =  (\cD_{\alpha_i} L(\lambda))_{\gg \gl-{\rm ss}} [i-1]$. By Lemma \ref{gl-n-loc-fin}  we have
$ (\cD_{\alpha_{i-1}} N)_{\gg \gl {\rm -ss}}[i] \subset L(\lambda)_{\gg \gl {\rm -ss}}$. Hence
$$
\dot{L}(\lambda)   \subset (\cD_{\alpha_{i-1}} \dot{L}(\mu))_{\gg \gl-{\rm ss}} [i] \subset  (\cD_{\alpha_{i-1}} L(\lambda'))_{\gg \gl-{\rm ss}} [i] \subset L(\lambda)_{\gg \gl-{\rm ss}}.
$$
Therefore,  $ (\cD_{\alpha_{i-1}} L(\lambda'))_{\rm ss} [i] $ is a $\gq (n)$-module containing $\dot{L}(\lambda) $ as a $\gg \gl_n$-subquotient and whose $\gg \gl_n$-semisimplification is a submodule of $L(\lambda)_{\gg \gl-{\rm ss}}$. Hence $ (\cD_{\alpha_{i-1}} L(\lambda'))_{\rm ss} [i] =L(\lambda)$. Now using the lemma  we show that
$$
\dot{L}(\lambda') \subset  (\cD_{\alpha_{i}} L(\lambda))_{\gg \gl-{\rm ss}} [i-1] \subset L(\lambda')_{\gg \gl-{\rm ss}},
$$
which eventually implies that $L(\lambda') = (\cD_{\alpha_i} L(\lambda))_{\rm ss} [i-1]$.
 \end{proof}

\subsection{} \begin{prop}{prop-link}   Let $\lambda'$ be a bounded integral weight and  $\xymatrix{
 \lambda \ar[r]^{i} & \lambda'}$, i.e. $L(\lambda)$  is a subquotient of $\cD_{\alpha_i} L(\lambda')$.

(i) If $\mu \in JH(\lambda)$ and $\mu'$ is bounded with $\xymatrix{
 \mu \ar@{-->}[r]^{i} & \mu'}$, then $\mu' \in JH(\lambda')$.

 (ii)  If $\mu$ is bounded of the same type as $\lambda$, $\mu' \in JH(\lambda')$ and $\xymatrix{
 \mu \ar@{-->}[r]^{i} & \mu'}$, then $\mu \in JH(\lambda)$.
\end{prop}
\begin{proof}
(i) Let $0 = L_0 \subset L_1 \subset...\subset L_m = L(\lambda')$ be a composition series of the $\gg \gl_n$-module $L(\lambda')$. Using the exactness of the localization functor, we obtain a series
$0 \subset \cD_{\alpha_i}L_1 \subset...\subset \cD_{\alpha_i}L_m = \cD_{\alpha_i}L(\lambda')$ whose quotients are $\cD_{\alpha_i} \left( L_j / L_{j-1}\right)$ (possibly zero). On one hand $\dot{L} (\mu)$ is a $\gg\gl_n$-subquotient of $L(\lambda)$, and on the other hand $L(\lambda)$ is a $\gq(n)$-subquotient of $\cD_{\alpha_i}L(\lambda')$. Therefore, there is $j\geq 1$ such that $\dot{L} (\mu)$ is a $\gg\gl_n$-subquotient of  $\cD_{\alpha_i} \left( L_j / L_{j-1}\right)$. But $L_j /L_{j-1}$ is a highest weight module and there is unique bounded weight $\eta$ for which $\xymatrix{
 \mu \ar@{-->}[r]^{i} & \eta}$. Therefore $\eta = \mu' \in JH(\lambda')$.

 (ii) Assume the contrary. Using the explicit description of the $\gg \gl_n$-families in \S \ref{gl-n-fam}, $\lambda$ and $\mu$ are  of type either $i-1$ or $i+1$. Assume that they are of type $i-1$ (the case of type $i+1$ is analogous). Because of our assumptions, there is a weight $\mu_0 \notin JH(\lambda)$ of type $i-1$ for which $\xymatrix{
 \mu_0 \ar@{-->}[r]^{i} & \mu_0'}$ and $\mu_0' \in JH(\lambda')$.  Let $\lambda_0$  be such that $L(\lambda_0)$ is a subquotient of $\cD_{\alpha_i} L(\lambda')$ and $\mu_0 \in JH(\lambda_0)$. In particular, $L(\lambda)$ and $L(\lambda_0)$ are two nonisomophic subquotients of type $i-1$ of $\cD_{\alpha_i} L(\lambda')$. This contradicts to Proposition \ref{prop-loc-fin}.
  \end{proof}

\begin{rem}{}
We conjecture that the above proposition can be generalized to all weights (not necessarily bounded) $\lambda$, $\lambda'$, $\mu$, and $\mu'$. This will be addressed in a future work.
\end{rem}

\section{Examples}

In this section we consider families of bounded modules of $\lambda$ for $\lambda := c \vareps_1$. The modules in these families are of ``small $\gg \gl_n$-length'' in the following sense:  they have $n$ or $n+1$ pairwise nonisomorphic $\gg \gl_n$-subquotients.  We are going to prove the following theorem (see~\S~\ref{gl-n-fam} for notations).

\subsection{} \begin{thm}{th-eps1}
Let $\lambda = c\vareps_1$,.

(i) For $c \in \Z_{>0}$ one has
\begin{eqnarray*}
L(s_{n-1}...s_1* \lambda)_{\gg \gl\rm{-ss}} & =&  \bigoplus_{i=1}^{n} \dot{L}(s_{n-1}...s_i \cdot s_{i-1}...s_1 * \lambda )^{\oplus 2},\\
L(s_{k}...s_1* \lambda)_{\gg \gl\rm{-ss}} & = &  \bigoplus \{ \dot{L}(\mu)^{\oplus 2} \; | \; \xymatrix{
 \mu \ar@{-->}[r]^{k} & \mu'}, \mu' \in JH(s_{k+1}...s_1* \lambda)\} \mbox{ for } k < n-1.
 \end{eqnarray*}

 (ii) For $c \in \Z_{<0}$ one has
\begin{eqnarray*}
L(s_1...s_{n-1}* \lambda)_{\gg \gl\rm{-ss}} & =&  \bigoplus_{i=1}^{n} \dot{L}(s_{1}...s_{i-1} \cdot s_{i}...s_{n-1} * \lambda )^{\oplus 2},\\
L(s_{k}...s_{n-1}* \lambda)_{\gg \gl\rm{-ss}} & = &  \bigoplus \{ \dot{L}(\mu)^{\oplus 2} \; | \; \xymatrix{
 \mu \ar@{-->}[r]^{i} & \mu'}, \mu' \in JH(s_{k-1}...s_{n-1}* \lambda)\} \mbox{ for } k > 1.
 \end{eqnarray*}

(iii) For $c=0$ one has
\begin{eqnarray*}
L(s_{k}* 0)_{\gg \gl \rm{-ss}} = \bigoplus_{i=1}^{n-1} \dot{L}\left(\prod_{j=k}^is_{j} \cdot  0 \right)^{\oplus 2}.
\end{eqnarray*}
(iv) For $c \notin \Z$ one has the same formulas as in (i). In particular,
\begin{eqnarray*}
L(s_{k}...s_1* \lambda)_{\gg \gl\rm{-ss}} & = &  \bigoplus \{ \dot{L}(s_{k+1}...s_{n-1} \cdot \mu)^{\oplus 2} \; | \; \mu \in JH(s_{n-1}...s_1 * \lambda) \} \mbox{ for } k < n-1.
 \end{eqnarray*}
 In the sums in (i) and (ii)  we use the following convention: $s_{n-1}...s_i \cdot \mu = s_i...s_{n-1} * \mu =  \mu$ if $i=n$ and $s_{i-1}...s_1 * \mu= s_1...s_{i-1} \cdot \mu = \mu$ if $i=1$.

\end{thm}

\subsection{Examples} Consider the case $\gq (4)$.

\subsubsection{}
Consider the weight $\lambda = 2 \vareps_1$. Then Theorem \ref{th-eps1} (i) implies
\begin{eqnarray*}
JH(s_3s_2s_1*\lambda) & = &  \{ s_3s_2s_1*\lambda, s_3\cdot s_2s_1*\lambda, s_3s_2\cdot s_1*\lambda, s_3s_2s_1 \cdot \lambda\}; \\
JH(s_2s_1*\lambda) & = &  \{ s_3 \cdot s_3s_2s_1*\lambda, s_2s_1*\lambda, s_2\cdot s_1*\lambda, s_2s_1 \cdot \lambda\}; \\
JH(s_1*\lambda) & = &  \{ s_1s_3 \cdot s_3s_2s_1*\lambda, s_2 \cdot s_2s_1*\lambda,  s_1*\lambda, s_1 \cdot \lambda\}.
\end{eqnarray*}

Recall that each of the $\gg \gl_n$-subquotients of $L(s_i...s_1* \lambda)$ has multiplicity $2$. Also,  $JH(\lambda) = \{s_1 \cdot s_1*\lambda, \lambda\}.$

\subsubsection{}
Consider the weight $\lambda = c \vareps_1$, $c \notin \Z$. From Theorem \ref{th-eps1} (iv) we have the following table.

$$\begin{tabular}{|l | l|}
\hline
$\gq(4)$-weights & Highest weights of the $\gg \gl_4$-submodules \\
\hline
$\lambda$ & $\lambda,\  s_1\cdot s_1 *\lambda,\  s_1s_2 \cdot s_2s_1 *\lambda,\  s_1s_2s_3 \cdot s_3s_2s_1*\lambda $ \\
\hline
$s_1 * \lambda$ & $s_1 \cdot \lambda, \ s_1 *\lambda,\  s_2 \cdot s_2s_1 *\lambda,\  s_2s_3 \cdot s_3s_2s_1*\lambda $ \\
\hline
$s_2s_1 * \lambda$& $s_2s_1 \cdot \lambda, \ s_2 \cdot s_1 *\lambda,\  s_2s_1 *\lambda,\  s_3 \cdot s_3s_2s_1*\lambda $ \\
\hline
$s_3s_2s_1 * \lambda$ & $s_3s_2s_1 \cdot \lambda, \ s_3s_2 \cdot s_1 *\lambda,\  s_3 \cdot s_2s_1 *\lambda,\  s_3s_2s_1*\lambda $\\
\hline
\end{tabular}$$
In particular, for $\mu = s_3s_2s_1 * \lambda = c \varepsilon_4$,
$$L(\mu) = \dot{L} (\mu)^{\oplus 2} \oplus  \dot{L} (\mu - \alpha_3)^{\oplus 2} \oplus \dot{L} (\mu - \alpha_2 - 2\alpha_3)^{\oplus 2} \oplus \dot{L} (\mu - \alpha_1 - 2\alpha_2 - 3\alpha_3)^{\oplus 2}. $$

\subsection{}\label{c-fam} The rest of this section is devoted to the proof of Theorem \ref{th-eps1}.

Retain notation of~\S~\ref{sec_arrow}. Observe that
the weight $\lambda = c \vareps_1$ is  regular integral if $c \in \Z$, and nonintegral otherwise. In the case $c \in \Z_{>0}$ (respectively, $c \in \Z_{<0}$) there are two regular integral families of $\lambda$ listed below - one of regularity $1$ (resp., $n-1$), and another of regularity $2, 3, ..., n$ (resp., $1,2,...,n-2$). The regular integral family of $\lambda$ of regularity $1$ for $c \in \Z_{>0}$ can be described by the graph
\begin{equation} \label{fam-pos}
s_{1}*\lambda \xrightarrow{2} s_2s_1*\lambda \xrightarrow{3}...  ...  \xrightarrow{n-1}s_{n-1}...s_1* \lambda.
\end{equation}
If $c \in \Z_{<0}$, then $L(\lambda)$ is a part of the regular integral family of $\lambda' = c \vareps_n = \iota (\lambda)$ of regularity $n-1$:
$$
s_{1}...s_{n-1}* \lambda' \xleftarrow{1} s_{2}...s_{n-1}*\lambda' \xleftarrow{3}...  ...  \xleftarrow{n-2} s_{n-1}* \lambda'
$$
 In addition we have $\lambda \xrightarrow{1} s_1*\lambda$ and $ s_{n-1}* \lambda' \xleftarrow{n-1} \lambda' $.

If $c =0$ then the regular integral family of $\lambda$ of regularity $1,2,...,n-1$:
\begin{equation} \label{fam-zero}
\xymatrix{s_1 * \lambda
 \ar@<0.5ex>[r]^2 & s_2 * \lambda \ar@<0.5ex>[l]^1
\ar@<0.5ex>[r]^{3 \! \! }& \ar@<0.5ex>[l]^{2\! \! \!} ...   \ar@<0.5ex>[r]^{n-1 \hspace{.5cm}}
&  \ar@<0.5ex>[l]^{n-2 \hspace{.5cm}} s_{n-1}* \lambda}
\end{equation}
Finally, in the case $c \notin \Z$, the nonintegral family of the bounded nonitegral weight $\lambda$ of type $1$ is
\begin{equation} \label{fam-nonint}
\xymatrix{ \lambda
 \ar@<0.5ex>[r]^{\! \! \! \! \! \! 1} & s_1 * \lambda  \ar@<0.5ex>[l]^{\! \! \! \! \! \!1}
\ar@<0.5ex>[r]^{2 \! \! \! \!}& \ar@<0.5ex>[l]^{2\! \! \! \!} ...   \ar@<0.5ex>[r]^{n-1 \hspace{1cm}}
&  \ar@<0.5ex>[l]^{n-1 \hspace{1cm}} s_{n-1}...s_1* \lambda}
\end{equation}
In what follows we will describe the $\gg \gl_n$-subquotients of the modules in the four families above.

\subsection{}  We will use  a degree formula for the $\gg \gl_n$-bounded modules of highest type. Considering a weight $\mu$ as a weight of ${\gg \gl_{n-1}\times \gg \gl_1} $, by $ \dot{L}_{\gg \gl_{n-1}\times \gg \gl_1} (\mu)$ we denote the corresponding simple $(\gg \gl_{n-1}\times \gg \gl_1)$-module.
The proof of the following proposition follows from  the fact that the parabolically induced module  from  $ \dot{L}_{\gg \gl_{n-1}\times \gg \gl_1} (\mu)$ is simple if $\mu$ is
nonintegral or singular and the module has length $2$ if $\mu$ is regular integral (see Lemma 11.2 in \cite{M} for details).

\subsubsection{}\begin{prop}{deg-for}
Let $\mu$ be a $\gg \gl_n$-bounded weight of type $n-1$.

 If $\mu $ is singular  or nonintegral, then
$\deg \dot{L} (\mu) = \dim \dot{L}_{\gg \gl_{n-1}\times \gg \gl_1} (\mu)$.

If $\mu$ is regular integral, then $\mu = s_{n-1}...s_k \cdot \eta$  where $\eta$ is a $\gg \gl_n$-dominant integral weight and one has
$$
\deg \dot{L} (s_{n-1}...s_k \cdot \eta) = \sum_{i\geq k} (-1)^{i-k}\dim \dot{L}({s_{n-1}...s_i \cdot \eta})
$$

\end{prop}

\subsection{The case of $c \in \Z$, $c \neq 0$}

Since  $\dot{L} (\mu)$ is a $\gg \gl_n$-subquotient of $L(s_i...s_{n-1} * \lambda)$ if and only if $\dot{L} (\iota(\mu))$ is a $\gg \gl_n$-subquotient of $L(s_1...s_{n-i} * \iota (\lambda))$, it is enough to consider just the case $c \in \Z_{>0}$.

\subsubsection{} \begin{lem}{lem-pos}
Let $\lambda = c\vareps_1$, $c \in \Z_{>0}$. Then $s_{n-1}...s_i \cdot s_{i-1}...s_1 * \lambda \in JH (s_{n-1}...s_1*\lambda)$ for every $i = 1,...,n-1$.
\end{lem}
\begin{proof}
We use extensively Proposition \ref{prop-link} (i). We prove by induction a stronger statement:   for every $k=1,...,n-1$, $s_{k}...s_i \cdot s_{i-1}...s_1 * \lambda \in JH (s_{k}...s_1*\lambda)$ for every $i = 1,...,k$. For $k=1$,  we use that $\lambda \xrightarrow{1} s_1*\lambda$ by \S \ref{c-fam} and that $\xymatrix{ \lambda
\ar@{-->}[r]^{1} & s_1 \cdot \lambda  }$ by \S \ref{gl-n-fam}.  Since $\lambda \in JH (\lambda)$, Proposition \ref{prop-link} (i) implies that $s_1 \cdot \lambda \in JH (s_1*\lambda)$. This together with $s_1 * \lambda \in JH (s_1*\lambda)$ proves the case $k=1$. Assume that  $s_{k}...s_i \cdot s_{i-1}...s_1 * \lambda \in JH (s_{k}...s_1*\lambda)$.  It is not difficult to check that $s_{k+1}...s_i \cdot s_{i-1}...s_1 * \lambda \neq s_{k}...s_i \cdot s_{i-1}...s_1 * \lambda$. Thus we  have
$$
s_{k}...s_1*\lambda \xrightarrow{k+1} s_{k+1}...s_1*\lambda; \;
\xymatrix{ s_{k}...s_i \cdot s_{i-1}...s_1 * \lambda
\ar@{-->}[r]^{k+1} & s_{k+1}...s_i \cdot s_{i-1}...s_1 * \lambda}
$$
Now by the induction hypothesis and Proposition \ref{prop-link} (i) we obtain
$s_{k+1}...s_i \cdot s_{i-1}...s_1 * \lambda \in JH (s_{k+1}...s_1*\lambda)$ for $i=1,...,k$. The case $i=k+1$ is obvious.
\end{proof}

\subsubsection{} Recall that for $i =n$, $s_{n-1}...s_i \cdot s_{i-1}...s_1 * \lambda:= s_{n-1}...s_1 * \lambda.$

 \begin{lem}{deg-pos}
Let $\lambda = c\vareps_1$, $c \in \Z_{>0}$.  Then $\sum_{i=1}^n \deg \dot{L}(s_{n-1}...s_i \cdot s_{i-1}...s_1 * \lambda ) = 2^{n-1}$. In particular, $\deg L(s_{n-1}...s_1*\lambda) \geq 2^n$, and equality holds if and only if $L(s_{n-1}...s_1* \lambda)_{\gg \gl\rm{-ss}} = \bigoplus_{i=1}^{n} \dot{L}(s_{n-1}...s_i \cdot s_{i-1}...s_1 * \lambda )^{\oplus 2}$.
\end{lem}
\begin{proof}
It is easy to check that $s_{n-1}...s_i \cdot s_{i-1}...s_1 * \lambda = (0,...,0,-1,...,-1,c+i)$ (with $i-1$ many $``-1''$) is either a regular integral weight of type $n-1$ in a family of regularity $1$, or is a singular weight of type $n-1$.  By Proposition \ref{deg-for},
$$
\sum_{i=1}^n \deg \dot{L}(s_{n-1}...s_i \cdot s_{i-1}...s_1 * \lambda ) = \sum_{i=1}^n \binom{n-1}{i-1} = 2^{n-1}
$$
For the inequality in the lemma we use Lemma \ref{lem-pos} and the fact   that every $\gg \gl_n$-subquotient of $L(s_{n-1}...s_1*\lambda)$
comes with multiplicity at least $2$, see~\Lem{lemcliff}. Now, since all $s_{n-1}...s_i \cdot s_{i-1}...s_1 * \lambda $ are of $\gg\gl_n$-type $n-1$, and $s_{n-1}...s_1*\lambda$ is of $\gq(n)$-type $n-1$, by Proposition \ref{gl-bounded} and Corollary \ref{cor-q-bounded} we have that all $L(s_{n-1}...s_1*\lambda)$,  $ \dot{L}(s_{n-1}...s_i \cdot s_{i-1}...s_1 * \lambda )$, $i=1,...,n$, have the same shadow. Hence
$$
\deg L(s_{n-1}...s_1*\lambda) \geq 2 \sum_{i=1}^n \deg \dot{L}(s_{n-1}...s_i \cdot s_{i-1}...s_1 * \lambda ) = 2^n.
$$
 If equality holds, then all $\dot{L}(s_{n-1}...s_i \cdot s_{i-1}...s_1 * \lambda)$ form the complete set of simple $\gg \gl_n$-subquotients (each coming with multiplicity $2$) of $L(s_{n-1}...s_1*\lambda)$ having the same shadow as the one of $L(s_{n-1}...s_1*\lambda)$. On the other hand, all infinite dimensional  $\gg \gl_n$-subquotients  of $L(s_{n-1}...s_1*\lambda)$ have the same shadow, so it remains to show that $L(s_{n-1}...s_1*\lambda)$ has no finite dimensional $\gg \gl_n$-subquotients. This follows easily from the fact that the set of weights $\mu$ for which $\mu \leq s_{n-1}...s_1*\lambda = c \varepsilon_n$ contains no $\gg \gl_n$-dominant integral weights.
 \end{proof}

\subsubsection{} Recall the definition of the module ${\mathcal F}_{c}$ (Example \ref{fam_bounded}).

\begin{lem}{sub_fam}
Let $\lambda = c\vareps_1$, $c \in \Z_{>0}$. Then $L(s_k...s_1 * \lambda)$ is a subquotient of ${\mathcal F}_c$ for every $k=1,...,n-1$.
\end{lem}
\begin{proof}
Note that $F_c:={\mathcal F}_{c}\cap \mathbb{C}[x_1,\ldots,x_n,\xi_1,\ldots,\xi_n]$
is a submodule of $\mathcal{F}_c$. One readily sees that
  $\fn^+ (x_1^{-1}x_2^{c} \xi_1) \subset F_c$  so
   $v_1 := x_2^{c} \frac{\xi_1}{x_1}$ is an $\gn^{+}$-primitive vector in ${\mathcal F}_{c}/F_c$ of weight $(0,c,...,0) = s_1*\lambda$. This implies that $L(s_1* \lambda)$ is subquotient of  ${\mathcal F}_{c}$.  Let $k> 1$ and let  $v_k := x_1^{c-k}\xi_1...\xi_k$. Since  $\fn^+ (v_k) \subset U v_{k-1}$, it is not difficult to show that $v_k$ is an $\gn^{+}$-primitive vector in ${\mathcal F}_{c}/(Uv_1+...+Uv_{k-1})$. Using that the weight of $v_k$ is $s_k...s_1* \lambda$, we complete the proof by induction on $k$.
\end{proof}

\subsubsection{Proof of Theorem  \ref{th-eps1} (i)}\label{proof-pos} Retain notation of \S \ref{notation-ss}.

By Lemma \ref{sub_fam} we have that  $L(s_{n-1}...s_1*\lambda)$ is a subquotient of ${\mathcal F}_{c}$ and hence $\deg L(s_{n-1}...s_1*\lambda) \leq 2^n$.
But, by Lemma \ref{deg-pos}, $\deg L(s_{n-1}...s_1*\lambda) \geq 2^n$.
Therefore, we must have equalities, which, by the same lemma, implies the first identity of Theorem  \ref{th-eps1} (i). To prove the second identity we use Proposition
\ref{prop-link} (ii) recursively for $k=n-2,n-3,...1$.\qed

\subsection{The case of $c =0 $} In this case $\lambda = 0$.

\subsubsection{} \begin{lem}{lem-zero}
 $s_{n-1}...s_i \cdot 0 \in JH (s_{n-1}* 0)$ for every $i = 1,...,n-1$.
\end{lem}
\begin{proof}
For convenience we use $\lambda = 0$.  The proof follows the same reasoning as the one of Lemma \ref{lem-pos}.  Namely, we prove by induction a stronger statement:   for every $k=1,...,n-1$, $s_{k}...s_i \cdot \lambda \in JH (s_{k} *\lambda)$ for every $i = 1,...,k$. For $k=1$,  we use that $\lambda \xrightarrow{1} s_1*\lambda$ by \S \ref{c-fam} and that $\xymatrix{ \lambda
\ar@{-->}[r]^{1} & s_1 \cdot \lambda  }$ by \S \ref{gl-n-fam}. Proposition \ref{prop-link} (i) implies that $s_1 \cdot \lambda \in JH (s_1*\lambda)$,
which, together with   $s_1 * \lambda \in JH (s_1*\lambda)$ proves the case $k=1$. For the induction step we use that
$$
\xymatrix{s_k *  \lambda
 \ar@<0.5ex>[r]^{k+1} & s_{k+1} * \lambda  \ar@<0.5ex>[l]^{k} }; \;
\xymatrix{s_k..s_i \cdot  \lambda
 \ar@{-->}@<0.5ex>[r]^{k+1} & s_{k+1}...s_i \cdot  \lambda  \ar@{-->}@<0.5ex>[l]^{k} }
$$
Note that in this proof, in contrast to the proof of Lemma \ref{lem-pos},  we need just Proposition \ref{prop-link} (i),
because of the presence of double arrows in the family (\ref{fam-zero}). \end{proof}

\subsubsection{} \begin{lem}{deg-zero}
 $\sum_{i=1}^{n-1} \deg \dot{L}(s_{n-1}...s_i \cdot  0 ) = 2^{n-2}$. In particular, $\deg L(s_{n-1}*0) \geq 2^{n-1}$ and equality holds if and only if $L(s_{n-1}*0)_{\gg \gl\rm{-ss}}  = \bigoplus_{i=1}^n  \dot{L}(s_{n-1}...s_i \cdot  0 ) $.
\end{lem}
\begin{proof}
Let again $\lambda = 0$. In this case $s_{n-1}...s_i \cdot \lambda = (0,...,0,-1,...,-1, n-i)$ (with $n-i$ many $``-1''$). Using Proposition \ref{deg-for}, we have
$$
\deg \dot{L} (s_{n-1}...s_i \cdot \lambda) =  \binom{n-1}{i} - \binom{n-1}{i+1}+ \binom{n-1}{i+2}-...
$$
Hence
$$
\sum_{i=1}^{n-1} \deg \dot{L}(s_{n-1}...s_i \cdot  \lambda ) =(n-1) + \binom{n-1}{3} + ... = 2^{n-2}.
$$
To prove $\deg L(s_{n-1}*0) \geq 2^{n-1}$ we proceed like in the proof of Lemma \ref{deg-pos}.
Namely,  we use that $L(s_{n-1}*0)$ and all $\dot{L}(s_{n-1}...s_i \cdot  0 )$ have the same shadow. If equality holds, we have that $\dot{L}(s_{n-1}...s_i \cdot  0 )$ are all infinite dimensional simple $\gg \gl_n$-subquotients of $L(s_{n-1}*0)$. But since there are no $\gg \gl_n$ dominant integral weights $\mu$ such that $\mu \leq s_{n-1}*0 = (0,0,...,-1,1)$, $L(s_{n-1}*0)$ has no finite-dimensional $\gg \gl_n$-subquotients. \end{proof}

\subsubsection{}
Let $J:=\sum_{i=1}^n\left( x_i \frac{\partial}{\partial \xi_i} - \xi_i  \frac{\partial}{\partial x_i} \right)$. One has the following.

 (i)  $J$ commutes with all elements of $\gq (n)$.

(ii) $J({\mathcal F}_0)\subset {\mathcal F}_0$ and $(J|_{{\mathcal F}_0})^2 = 0$.

(iii) Let ${\mathcal F}_0^J := \Ker J|_{{\mathcal F}_0}$. Then ${\mathcal F}_0/{\mathcal F}_0^J \simeq {\mathcal F}_0^J$.

\subsubsection{}
\begin{lem}{sub_fam_zero}  $L(s_1 * 0)$ is a subquotient of ${\mathcal F}_0$ and hence of ${\mathcal F}_0^J$. In particular, $\deg L(s_{i}*0) \leq 2^{n-1}$, for $i=1,...,n-1$.
\end{lem}
\begin{proof}
To prove that $L(s_1 * 0)$ is a subquotient of ${\mathcal F}_0$, one uses that $\xi_2 x_1^{-1} - x_1^{-2}x_2 \xi_1$ is an $\gn^{+}$-primitive vector in ${\mathcal F}_{0}/{\mathbb C}$. Therefore $\deg L(s_{1}*0) \leq 2^{n-1}$. From (\ref{fam-zero}) we have that all $L(s_{i}*0)$ have the same degree, and hence $\deg L(s_{i}*0) \leq 2^{n-1}$.
\end{proof}

\subsubsection{Proof of Theorem \ref{th-eps1} (iii)}
We repeat the same reasoning as in \S \ref{proof-pos} and apply  the three preceding lemmas in this subsection to prove the statement for $L(s_{n-1} * \lambda)$. To complete the proof for arbitrary $L(s_{k} * \lambda)$ we apply (\ref{fam-zero}), \S \ref{gl-n-fam} (i), and Proposition \ref{prop-link}.
\qed

\subsubsection{Remark}
Another way to state Theorem \ref{th-eps1} (iii) is that all modules of type $k$ in the regular integral $\gg \gl_n$-families of $\lambda = 0$ form the complete set of $\gg \gl_n$-subquotients of the $\gq(n)$-module $L(s_k * \lambda)$.

\subsection{The case of $c \notin \Z$} In this case we proceed like in the case $c \in \Z_{>0}$. The presence of double arrows makes the reasoning easier.  Recall that  $s_{n-1}...s_i \cdot s_{i-1}...s_1 * \lambda:= s_{n-1}...s_1 * \lambda$ for $i=n$ and $s_{k}...s_1* \lambda = \lambda$ for $k=0$.

\subsubsection{} \begin{lem}{lem-nonint}
Let $\lambda = c\vareps_1$, $c \notin \Z$.

(i) $s_{n-1}...s_i \cdot s_{i-1}...s_1 * \lambda \in JH (s_{n-1}...s_1*\lambda)$ for every $i = 1,...,n$.

(ii) $\sum_{i=1}^n \deg \dot{L}(s_{n-1}...s_i \cdot s_{i-1}...s_1 * \lambda ) = 2^{n-1}$. In particular, $\deg L(s_{n-1}...s_1*\lambda) \geq 2^n$ and equality holds if and only if $L(s_{n-1}...s_1*\lambda)_{\gg \gl\rm{-ss}}  = \bigoplus_{i=1}^n \dot{L}(s_{n-1}...s_i \cdot s_{i-1}...s_1 * \lambda)$.

(iii) $L(\lambda)$ is a subquotient of ${\mathcal F}_c$.

\end{lem}

\subsubsection{Proof of Theorem \ref{th-eps1} (iv)} We apply Lemma \ref{lem-nonint} and reason as in \S \ref{proof-pos}. From (\ref{fam-pos}) we observe that all modules in the nonintegral family of $\lambda$ have the same degree. Using this and Lemma \ref{lem-nonint} (iii) we find $\deg L(s_{n-1}...s_1*\lambda) \leq 2^n$. Now with the aid of Lemma \ref{lem-nonint} (ii)  and the fact that  the simple $\gg \gl_n$-subquotients of $L(s_k...s_1* \lambda)$ have distinct central characters we complete the proof.
\qed

\section{ Classification of simple cuspidal $\gq(n)$--modules} \label{sec-cusp}

In this section we reduce the classification of simple weight ${\mathfrak q}(n)$-modules to the classification of simple highest weight bounded ${\mathfrak q}(n)$-modules of type $1$. The former classification  is first reduced the to the classification of simple cuspidal ${\mathfrak q}(n)$-modules with a Fernando-Futorny parabolic reduction theorem. Then copying methods of \cite{M} we present every cuspidal module as a twisted localization of  a highest weight bounded module. 

\subsection{} One of the main theorems  in \cite{DMP} states that every simple weight module of a simple finite dimensional Lie superalgebra $\gg$ can be presented in a unique way (up to a Weyl group conjugacy) as a parabolically induced module from a cuspidal module. We refer the reader to Theorem 6.1 in \cite{DMP} for details. This theorem is a generalization of the so called Fernando-Futorny parabolic induction theorem in the case when $\gg$ is a Lie algebra. In the case $\gg = \gp \gs \gq (n)$ (or, equivalently for $\gq (n)$) using the description of the so called ``cuspidal Levi subsuperalgebras'' one has the following $\gq (n)$-version of the result of Dimitrov-Mathieu-Penkov (for sake of simplicity the uniqueness part is omiited).

\begin{thm}{} Every simple weight $\gq(n)$-module is parabolically induced from a cuspidal module over $\gq (n_1) \oplus....\oplus \gq(n_k)$, for some positive integers $n_1,...,n_k$ with $n_1+...+n_k = n$.
\end{thm}

We present a short proof of the above theorem based on the Fernando-Futorny parabolic induction theorem which was kindly suggested by the referee. 

Let $L$ be a simple non-cuspidal ${\mathfrak q}(n)$-module. As $U({\mathfrak q}(n))$ is finite over $U(\mathfrak{gl}_n)$, the module $L$ is finitely generated as a $\mathfrak{gl}_n$-module and hence has a simple top, say $\dot{L}$. The module $\dot{L}$ is  not cuspidal, hence, it is parabolically induced from some parabolic subalgebra ${\mathfrak p}$ of $\gg \gl_n$ by the theorem of Fernando-Futorny. By adjunction between ${\rm Ind}^{{\mathfrak q}(n)}_{\mathfrak{gl}(n)}$ and ${\rm Res}_{{\mathfrak q}(n)}^{\mathfrak{gl}(n)}$, the module $L$ is a submodule of $\mbox{Ind}^{{\mathfrak q}(n)}_{\mathfrak{gl}(n)}\dot{L}$. The latter is just tensoring $L$ with a  finite dimension module. Hence $L$ is parabolically induced from the ``${\mathfrak q}(n)$-version'' of ${\mathfrak p}$.

\subsubsection{Remark}
The original theorem in \cite{DMP} is for the simple Lie superalgebra $\gp \gs \gq (n)$,
which is the simple subquotient of $\gq(n)$,
but the proof there can be easily modified for $\gq (n)$ as well.

\subsection{} The following result is a standard property of the localization functor, but for the reader's convenience a short proof is included.

\begin{lem}{}\label{lemma-loc} Let $\gg = \gq (n)$ or $\gg = \gg \gl_n$ and $\alpha \in \Delta$. If $M$ is a $\gg$--module and  $L$ is a simple submodule of
 ${\cD}_{\alpha} M$, then $L \subset M$.
\end{lem}
\begin{proof}
Let $x \in L$. There is $N\geq 0$ such that $f_{\alpha}^N x \in M$. Therefore $L \cap M $ contains $f_{\alpha}^N x$ and hence is nonzero. From the simplicity of $L$ we obtain $L \cap M = L$.
\end{proof}

\subsection{}
The following theorem is valid for $\gg \gl_n$--modules, see \cite{M}.

 \begin{thm}{cuspidal}
 Let $M$ be a simple cuspidal $\gq(n)$--module. Then there is a unique bounded weight $\lambda$ of type $1$,
which is  either regular integral, or singular, or nonintegral, and a unique tuple
$(x_1,...,x_{n-1})$ of $n-1$ complex nonintegral numbers such that
 $M \simeq \cD_{\alpha_1}^{x_1}  \cD_{\alpha_1+\alpha_2}^{x_2}...
\cD_{\alpha_1+...+\alpha_{n-1}}^{x_{n-1}} L(\lambda)$.
 \end{thm}

\subsection{Proof of existence of Theorem \ref{cuspidal}} Retain notation of  \S \ref{gencon}. We  introduce some notation that will be used both for the existence and the uniqueness parts of the theorem. Set  $\Sigma = \{ \alpha_1, \alpha_1 + \alpha_2,...,\alpha_1 +...+\alpha_{n-1}\}$, and denote by $\cD_{\Sigma}$ the functor $\cD_{\alpha_1}...\cD_{\alpha_1+...\alpha_{n-1}}$ on the category of all weight $\gq(n)$-modules. Also, for a weight $\mu = x_1 \alpha_1 +...+ x_{n-1} \alpha_{n-1}$, let
$\Phi_{\Sigma}^{\mu}:=\Phi_{\alpha_1}^{x_1}...\Phi_{\alpha_1+...+\alpha_{n-1}}^{x_{n-1}}$. Notice that since $[f_{\alpha}, f_{\beta}] = 0$ for every $\alpha, \beta \in \Sigma$, we can switch the order of the factors in the definitions of $\cD_{\Sigma}$ and $\Phi_{\Sigma}^{\mu}$.  Finally,  set $\cD_{\Sigma}^{\mu} = \Phi_{\Sigma}^{\mu} \cD_{\Sigma} =  \cD_{\alpha_1}^{x_1}  \cD_{\alpha_1+\alpha_2}^{x_2}...
\cD_{\alpha_1+...+\alpha_{n-1}}^{x_{n-1}} $.

Let $M_0$ be a simple $\gg \gl_n$-submodule of $M$.  Since $M_0$ is cuspidal, and the theorem holds for $\gg \gl_n$, there is a weight $\mu = x_1 \alpha_1 +...+ x_{n-1} \alpha_{n-1}$ and a $\gg \gl_n$--type 1 weight $\lambda_0$ such that $M_0 = \cD_{\Sigma}^{\mu} \dot{L} (\lambda_0)$. Therefore, the $\gg \gl_n$-module $\Phi_{\Sigma}^{-\mu} M$ contains $ \cD_{\Sigma}  \dot{L} (\lambda_0)$ as a submodule, and, in particular, has a vector $v$ such that $e_{\alpha} v = 0$ for every $\alpha \in \Delta^{\rm{fin}} \dot{L}(\lambda_0)$. Since $\ad(e_{\alpha})$ is nilpotent in $U(\mathfrak{q}(n))$ for every $\alpha \in \Delta$,   
$$
N = \{ x \in \Phi_{\Sigma}^{-\mu} M \; | \; e_{\alpha}^Nx = 0, \mbox{ for every } \alpha \in \Delta^{\rm{fin}} \dot{L}(\lambda_0), \mbox { and } N>>0\}
$$
is a $\mathfrak{q}(n)$-module. Moreover, $v \in N$ implies that $N$ is a  nontrivial bounded module. Since every finitely generated bounded module has finite length (see \S \ref{gl-roots}), we may fix a simple submodule $L$ of $N$. From $\cD_{\Sigma}^{\mu}L \subset M$ and the simplicity of $M$ we find that $\cD_{\Sigma}^{\mu}L = M$. On the other hand,  $ \Delta^{\rm{fin}} \dot{L}(\lambda_0) \subset  \Delta^{\rm{fin}} L $,  $\Sigma \subset  \Delta^{\rm{inj}} L$, and $ \Delta^{\rm{fin}}  \dot{L}(\lambda_0)  \sqcup \Sigma = \Delta$, imply $ \Delta^{\rm{inj}} L = \Sigma$ and  $ \Delta^{\rm{fin}} L = \Delta \setminus \Sigma$. Hence, $L = L (\lambda)$ for some type 1 weight $\lambda$.

\subsection{Proof of uniqueness of Theorem \ref{cuspidal}} With the notation of the previous subsection, assume that $M = \cD_{\Sigma}^{\mu_1} L(\lambda_1) =  \cD_{\Sigma}^{\mu_2} L(\lambda_2) $ for some weights $\mu_i$ and type-1 weights $\lambda_i$. Let again $S_0$ be a simple (cuspidal) $\gg \gl_n$-submodule of $S$.  Let for $i=1,2$, $L_i$ be a simple $\gg \gl_n$-submodule of $\Phi_{\Sigma}^{-\mu_i} S_0$. Since $L_i \subset \Phi_{\Sigma}^{-\mu_i} S_0 \subset \cD_{\Sigma} L(\lambda_i) $, by Lemma  \ref{lemma-loc}, $L_i \subset L(\lambda_i)$. Therefore $L_i = \dot{L} (\nu_i)$ some weights $\nu_i$. But then $\cD_{\Sigma}^{\mu_i} \dot{L}(\nu_i) \subset S_0$ and from the simplicity of $S_0$ we obtain $\cD_{\Sigma}^{\mu_1} \dot{L}(\nu_1) =  \cD_{\Sigma}^{\mu_2} \dot{L}(\nu_2) = S_0$. The uniqueness of the $\gg \gl_n$-version of the theorem implies $\mu_1 = \mu_2$. Now using that $\cD_{\Sigma} L(\lambda_1) = \cD_{\Sigma} L(\lambda_2)$ contains $L(\lambda_1)$ and $ L(\lambda_2)$ as submodules, and applying Lemma \ref{lemma-loc}, we verify that $L(\lambda_1) = L (\lambda_2)$ which completes the proof.

\section{Appendix}
Retain notation of Section~\ref{i-arrow}. Recall that $a\succ b$ if and only if $a>b$ or $a=b=0$.
Let $\lambda$ be a weight for $\gq (n)$.
\subsection{}
The following lemma shows that for a maximal regular weight $\lambda$
the weights $(s_i\ldots s_1)*\lambda$ form an increasing string.

\begin{lem}{lems1s2}
Assume that $s_j*\lambda<\lambda \text{ for } j=i,i+1,\ldots,k$. Then
$$
s_is_{i+1}\cdots s_k*\lambda<s_{i+1}\cdots s_k*\lambda<\lambda.
$$
\end{lem}
\begin{proof}
By the assumption, $(\lambda,\vareps_j)\succ (\lambda,\vareps_{j+1})$
for  $j=i,i+1,\ldots,k$.
For each $\nu$ one has $(s_j*\nu,\vareps_j)\in\{(\nu,\vareps_{j+1}), (\nu,\vareps_{j+1})-1\}$.
By induction for $i<k$ we obtain
$$(s_{i+1}\cdots s_k*\lambda,\vareps_{i+1})\leq (\lambda,\vareps_{k+1})$$
with the strict inequality if $(\lambda,\vareps_k)=(\lambda,\vareps_{k+1})=0$.

Since  $(\lambda,\vareps_i)\succ (\lambda,\vareps_{k+1})$ one has  $(\lambda,\vareps_i-\vareps_{k+1})>0$
or  $(\lambda,\vareps_i)=(\lambda,\vareps_{k+1})=0$.
If  $(\lambda,\vareps_i-\vareps_{k+1})>0$, we obtain
$$(s_{i+1}\cdots s_k*\lambda,\vareps_i-\vareps_{i+1})\geq (\lambda,\vareps_i-\vareps_{k+1})>0.$$
If  $(\lambda,\vareps_i)=(\lambda,\vareps_{k+1})=0$, then $(\lambda,\vareps_{k})=0$; by above, this gives
$$(s_{i+1}\cdots s_k*\lambda,\vareps_{i+1})<(\lambda,\vareps_{k+1})$$
so
$$(s_{i+1}\cdots s_k*\lambda,\vareps_i-\vareps_{i+1})>(\lambda,\vareps_i-\vareps_{k+1})\geq 0.$$
Thus in both cases one has
$$(s_{i+1}\cdots s_k*\lambda,\vareps_i-\vareps_{i+1})>0.$$

Note that for an integral weight $\nu$ the inequality  $(\nu,\alpha_i)>0$ forces $s_i*\nu<\nu$.
This establishes the inequality $s_is_{i+1}\cdots s_k*\lambda<s_{i+1}\cdots s_k*\lambda$ as required.
\end{proof}

\subsubsection{}
The following result  was used in the proof of~\Thm{thmq}.

\begin{lem}{eqqq}
Let $\lambda$ be an integral weight satisfying
$$(\lambda,\vareps_1)>(\lambda,\vareps_2)>\ldots >(\lambda,\vareps_{n-1})\succ (\lambda,\vareps_n).$$
Then $(s_1\ldots s_{n-1}*\lambda,\vareps_n)\in\{(\lambda,\vareps_{n-1}),(\lambda,\vareps_{n-1})+1\}$ and
$$(s_1\ldots s_{n-1}*\lambda,\vareps_2)>(s_1\ldots s_{n-1}*\lambda,\vareps_3)>\ldots>
(s_1\ldots s_{n-1}*\lambda,\vareps_n)  .$$
\end{lem}
\begin{proof}
The proof is by induction on $n$.
For $n=2$ we have $b_2\in\{a_1,a_1+1\}$ as required.
Suppose the claim holds for $n-1$.
Write
$$\lambda=:\sum_{i=1}^{n} a_i\vareps_i=(a_1,\ldots,a_{n}).$$

By the assumption,
$$a_1>a_2>\ldots >a_{n-1}\succ a_{n}.$$
Write
$$(b_1,\ldots,b_{n}):=s_1\ldots s_{n-1}*(a_1,\ldots,a_{n})$$
and notice that
$$(b_1,\ldots,b_{n-1})=s_1\ldots s_{n-2}*(a_1,\ldots,a_{n-2},b),$$
where $(b,b_{n})=s*(a_{n-1},a_{n})$ (we view $(a_{n-1},a_{n})$ as a weight for $\gq(2)$
and $s$ is the generator of $\widetilde{W}$ for  $\gq(2)$).
Since $b\in\{a_n,a_n-1\}$ one has
$$a_1>a_2>\ldots >a_{n-2}>b.$$
Thus, by induction hypothesis,
$$b_2>b_3>\ldots>b_{n-1},\ \ b_{n-1}\in\{a_{n-2},a_{n-2}+1\}.$$
Clearly, $b_n\in \{a_{n-1},a_{n-1}+1\}$. It remains to verify that $b_{n-1}>b_n$.
Assume that $b_{n-1}\leq b_n$.
Since $a_{n-1}>a_{n-2}$, this implies $b_n=a_{n-1}+1=a_{n-2}=b_{n-1}$ and $a_{n-1}+a_n=0$.
Then $(a_{n-2},a_{n-1}, a_n)=(a_{n-1}+1,a_{n-1},-a_{n-1})$
and thus $(b_{n-1},b_n)=(a_{n-1}+2,a_{n-1}+1)$, a contradiction.
 \end{proof}


\begin{thebibliography}{99}
\bibitem[Bb]{Bb} N.~Bourbaki, Groupes et alg\'ebres de Lie, Ch. IV-VI, Hermann, Paris, 1968.

\bibitem[Br]{Br} J. Brundan, Kazhdan-Lusztig polynomials and character formulae for the Lie superalgebra $\gq(n)$, {\it Adv. Math.} {\bf 182} (2004), 28--77.

\bibitem[De]{De}
V. Deodhar, On a construction of representations and a problem of Enright,
{\it Invent. Math.} {\bf 57} (1980), 101--118.

\bibitem[DMP]{DMP}
I. Dimitrov, O. Mathieu, I. Penkov, On the structure of weight
modules,  {\it Trans. Amer. Math. Soc.} {\bf 352} (2000),
2857--2869.

\bibitem[Fe]{Fe} S. Fernando, Lie algebra modules with finite-dimensional weight
spaces I, {\it Trans. Amer. Math. Soc.}  {\bf 322} (1990),
757-781.

\bibitem[FM]{FM} A. Frisk, V. Mazorchuk, Regular strongly typical blocks of $\mathcal{O}^{\mathfrak{q}}$, {\it Commun. Math. Phys.} {\bf 291} (2009),  533--542.


\bibitem[Fu]{Fu} V. Futorny, The weight representations of semisimple finite-dimensional Lie algebras, Ph. D. Thesis, Kiev University, 1987.

\bibitem[Go]{Go} M. Gorelik, Shapovalov determinants of $Q$-type Lie superalgebras, {\it
Int. Math. Res. Pap.}, Article ID {\bf 96895} (2006), 1--71.

\bibitem[Gr]{Gr} D. Grantcharov,  Explicit realizations of simple weight modules of classical Lie superalgebras, {\it Contemp. Math.} {\bf 499} (2009) 141--148.

\bibitem[K]{K} V. Kac, Lie superalgebras, {\it Adv. Math.} {\bf 26} (1977), 8--96.

\bibitem[M]{M}
O. Mathieu, Classification of irreducible weight modules, {\it Ann. Inst. Fourier} {\bf 50} (2000), 537--592.

\bibitem[Maz]{Maz} V. Mazorchuk, {\em Classification of simple $\gq(2)$-supermodules}, Tohoku Math. J. (2) 62 (2010),
no. 3, 401--426.




\bibitem[P]{P} I. Penkov, Characters of typical irreducible finite-dimensional $\gq(n)$-modules, {\it Funct. Anal.
Appl.} {\bf 20} (1986), 30--37.


\bibitem[PS1]{PS1} I. Penkov, V.  Serganova, Characters of irreducible $G$-modules and cohomology of $G/P$ for the Lie supergroup $G=Q(N)$,
{\it J. Math. Sci.} (New York) {\bf 84} (1997), 1382--1412.

\bibitem[PS2]{PS2} I. Penkov, V. Serganova,  Characters of finite-dimensional irreducible $\gq (n)$-modules,
{\it Lett. Math. Phys.} {\bf 40} (1997), 147--158.

\bibitem[Se]{Se} A. Sergeev, Tensor algebra of the identity representation as a module over the Lie superalgebras ${\rm Gl}(n,\,m)$ and $Q(n)$ (Russian),  {\it Mat. Sb. (N.S.)} {\bf 123(165)} (1984), 422--430.


\end{thebibliography}
\end{document}